\documentclass[11pt]{article}\textwidth 160mm\textheight 235mm
\usepackage{amsfonts}
\usepackage{amssymb}
\usepackage{graphicx}
\usepackage{amsmath}
\usepackage{amsthm}
\usepackage{enumitem}
\usepackage{tikz}
\usepackage{cancel}
\usepackage{nicematrix}
\usepackage{todonotes}
\usetikzlibrary{matrix}
\usetikzlibrary{arrows,shapes}
\usepackage{cite}
\usepackage{hyperref}
\hypersetup{colorlinks=true, urlcolor=blue}
\expandafter\let\expandafter\oldproof\csname\string\proof\endcsname
\let\oldendproof\endproof
\renewenvironment{proof}[1][\proofname]{%
  \oldproof[\ttfamily \scshape \bf #1. ]%
}{\oldendproof}
\def\O{{\bf O}}

\def\S{{\bf {S}}}

\def\B{{\bf B}}
\def\R{{\bf {R}}}
\def\N{{\bf {N}}}
\def\ox{\bar{x}}

\def\oz{\bar{z}}
\def\ov{\bar{v}}
\def\ow{\bar{w}}

\def\ss{\scriptsize }

\def\ve{\varepsilon}

\def\X{{\bf X}}
\def\P{{\bf P}}
\def\Y{{\bf Y}}
\def\hat{\widehat}

\def\dist{{\rm dist}}

\def\Lm{{\Lambda}}
\def\tto{\rightrightarrows}

\def\d{{\rm d}}
\def\l{{\ell}}
\def\sub{\partial}

\def\disp{\displaystyle}
\def\Hat{\widehat}

\def\Bar{\overline}
\def\ra{\rangle}
\def\la{\langle}
\def\ve{\varepsilon}

\def\tr{\mbox{\rm tr}\,}
\def\inte{\mbox{\rm int}\,}

\def\epi{\mbox{\rm epi}\,}

\def\dom{\mbox{\rm dom}\,}

\def\ker{\mbox{\rm ker}\,}

\def\dn{\downarrow}

\def\ph{\varphi}

\def\oR{\Bar{\bf{R}}}

\def\lm{\lambda}

\def\dd{\delta}
\def\al{\alpha}
\def\Th{\Theta}
\def\th{\theta}

\def\sm{\hbox{${1\over 2}$}}

\def\sce{\setcounter{equation}{0}}
\def\Diag{\mbox{\rm Diag}\,}

\def\prox{\mbox{\rm prox}}

\def\verl{ \;\rule[-0.4mm]{0.2mm}{0.27cm}\;}
\def\verll{ \;\rule[-0.7mm]{0.2mm}{0.37cm}\;}
\def\verlm{ \;\rule[-0.5mm]{0.2mm}{0.33cm}\;}

\begin{document}
\vspace*{0.5in}
\begin{center}
{\bf PARABOLIC REGULARITY OF SPECTRAL FUNCTIONS}\\[1 ex]
ASHKAN MOHAMMADI,\footnote{Department of Mathematics, Georgetown University, Washington, D.C. 20057, USA (ashkan.mohammadi@georgetown.edu).}
and  EBRAHIM SARABI\footnote{Department of Mathematics, Miami University, Oxford, OH 45065, USA (sarabim@miamioh.edu).
Research of this    author is partially supported by the U.S. National Science Foundation  under the grant DMS 2108546.}
\end{center}
\vspace*{0.05in}
\small{\bf Abstract.}  This paper is devoted to the study of the second-order variational analysis of  spectral functions. It is well-known that  
spectral functions can be expressed as a composite function of symmetric functions and eigenvalue functions.  We establish several second-order properties of spectral functions when their associated symmetric functions enjoy these properties. 
Our main attention is given  to characterize parabolic regularity for this class of functions. It was observed recently that parabolic regularity can play a central rule in ensuring the validity of important 
second-order variational properties such as twice epi-differentiability. We demonstrates that for convex spectral functions, their parabolic regularity amounts to that of their symmetric functions.
As an important consequence, we calculate the second subderivative of convex spectral functions, which allows us to establish second-order optimality conditions for a class of matrix optimization problems. \\[1ex]

\newtheorem{Theorem}{Theorem}[section]
\newtheorem{Proposition}[Theorem]{Proposition}
\newtheorem{Remark}[Theorem]{Remark}
\newtheorem{Lemma}[Theorem]{Lemma}
\newtheorem{Corollary}[Theorem]{Corollary}
\newtheorem{Definition}[Theorem]{Definition}
\newtheorem{Example}[Theorem]{Example}
\newtheorem{Algorithm}[Theorem]{Algorithm}
\renewcommand{\theequation}{{\thesection}.\arabic{equation}}
\renewcommand{\thefootnote}{\fnsymbol{footnote}}
\newcommand{\bigzero}{\mbox{\normalfont\Large\bfseries 0}}
\normalsize

\section{Introduction}\sce
This paper aims to study  second-order variational properties, including  parabolic regularity and twice epi-differentiability,  of spectral functions. These are functions 
$g:\S^n\to \oR:=[-\infty,\infty]$, where $\S^n$ stands for  the real vector space of $n\times n$ symmetric matrices,  that are orthogonally invariant, namely for any $n\times n$ symmetric   matrix $X$ and any $n\times n$ orthogonal matrix $U$, 
we have 
 $$
 g (X) = g (U^\top X U ).
 $$
 It is well-known (cf. \cite[Proposition~4]{l99}) that any spectral function $g$ can be equivalently expressed in a composite form 
 \begin{equation}\label{spec}
g(X)=(\th\circ\lm)(X), \;\; X\in \S^n,
\end{equation}
where $\th:\R^n\to  \oR$ is a permutation-invariant function on $\R^n$, called symmetric, and $\lm$ is a  function, which assigns to each matrix $X\in \S^n$ its eigenvalue vector $\big(\lm_1(X),\ldots,\lm_n(X)\big)$
arranged in nonincreasing order. 

Davis \cite{da} showed that convexity of the permutation-invariant function $\th$ in \eqref{spec} is inherited by the spectral function $g$. 
Similar observation was made by Lewis \cite{l96} about differentiability and strict differentiability   and by Lewis and Sendov \cite{ls} about twice differentiability. It was shown in \cite{l96, l99} and \cite{dlms}, respectively, 
that  the calculation of different notions of subdifferentials  of spectral functions  and    prox-regularity, which plays an important role in second-order variational analysis, 
enjoy this striking pattern as well.

The main question that we are trying to answer in this paper is whether such a striking pattern can be extended for other important second-order 
variational properties, including parabolic regularity  (see Definition~\ref{pre}) and twice epi-differentiability. While the former  was first  introduced  more than  two decades ago in \cite[Definition~13.69]{rw}, 
its persuasive role in second-order variational analysis was revealed quite recently in \cite{mms2, ms},  where it was shown for the first time that any parabolic regular function is twice epi-differentiable 
at any points in the graph of its subgradient mapping. This observation provided a systematic approach for  the study of twice epi-differentiability of extended-real-valued functions, which has important 
applications in  understanding various second-order variational concpets such as proto-differentiability of subgradient mappings (cf. \cite[Corollary~3.9]{ms}), twice epi-differentiability of the augmented Lagrangian functions associated 
with composite and constrained optimization problems (cf. \cite[Theorem~8.3]{mms2}), and the characterization of the quadratic growth condition for this class of functions (cf. \cite[Theorem~4.1]{hs}). Moreover,  parabolic regularity was utilized in  \cite{mms2, ms}
to obtain the {\em exact} chain rule for the second subderivative of certain composite functions, commonly seen in different classes optimization problems. Such a chain rule has an important application in 
finding  second-order necessary and sufficient optimality conditions for optimization problems.

It was demonstrated  in \cite{ms}  that 
important second-order variational properties of a composite function $\psi\circ F$, where $\psi:\X\to  (-\infty,\infty]$ is convex and $F:\X\to \Y$ is twice differentiable 
with $\X$  and $\Y$ being finite-dimensional Hilbert spaces, can be established at any $\ox\in \dom (\psi\circ F)$ provided that $\psi $ is parabolically regular and that 
the following metric subregularity constraint qualification is satisfied (cf. \cite[Definition~4.2]{ms}): There exists a constant $\kappa\ge 0$  such that 
 the  estimate 
\begin{equation} \label{mscq}
\dist \big( x, \dom (\psi\circ F)\big) \le \kappa\,\dist ( F(x),\dom \psi ) 
\end{equation}
 holds for all $x$ sufficiently close to $\ox$. Thus,  it is natural to ask  whether a similar approach can be utilized for the composite representation in \eqref{spec} of spectral functions.
To do so,   two major obstacles seem to hinder proceeding with the approach in \cite{ms}: 1) the lack of twice differentiability of  the inner function $\lm(\cdot)$ in \eqref{spec}; 2)
the validity of a constraint qualification similar to the aforementioned condition for the composite form  \eqref{spec}. 
Given the composite representation \eqref{spec}, 
it follows from    \cite[Proposition~2.3]{dlms} that   for any $X \in \S^n $, the equality 
\begin{equation} \label{specmsqc}
\dist ( X, \dom g) = \dist ( \lambda(X),\dom \theta )
\end{equation}
always holds. This simple but important observation from \cite{dlms} tells us that the required constraint qualification for dealing with the composite form \eqref{spec}
is automatically satisfied.
Moreover, looking closer into the established theory in \cite{mms2, ms} tells us that twice differentiability of the inner function was  not required. Indeed, a quadratic expansion will suffice to  
proceed in both these publications. Such a quadratic expansion, which is of a parabolic type, is already achieved in \cite{t1} for eigenvalue functions. 
These open the door for using the approach from \cite{mms2, ms} to study second-order variational properties of spectral functions. 

The outline of the paper is as follows. Section~\ref{sect02} recalls important notation and concepts related to the eigenvalue function. 
In Section~\ref{sect03}, we begin with establishing a chain rule for subderivative of the spectral function in \eqref{spec}. Section~\ref{sect04} 
is devoted to the study of the parabolic drivability  of spectral sets. We will obtain a chain rule for the parabolic subderivative, which plays a central role 
in the study of parabolic regularity of spectral functions. It is also shown that the parabolic subderivative is a symmetric function with respect to a subset of 
the space of orthogonal matrices. In  Section~\ref{sect05}, we demonstrate that the spectral function $g$ in \eqref{spec} is parabolically regular if and only if 
the symmetric function $\th$ in \eqref{spec} enjoys this property. As a consequence, we are going to calculate the second subderivative of spectral functions 
 when the symmetric functions associated with them are convex. This allows us to find second-order optimality conditions for a class of matrix optimization problems.

\section{Notation}\sce \label{sect02}
In what follows, $\X$ and $\Y$ are finite-dimensional Hilbert spaces.   By $\B$ we denote the closed unit ball in the space in question and by $\B_r(x):=x+r\B$ the closed ball centered at $x$ with radius $r>0$. For any set $C\subset \X$, its indicator function is defined by $\dd_C(x)=0$ for $x\in C$ and $\dd_C(x)=\infty$ otherwise. We denote by $\dist (x,C)$  the distance between $x\in \X$ and a set $C$.
For $v\in \X$, the subspace $\{w\in \X\, |\, \la w,v\ra=0\}$ is denoted by $[v]^\bot$. 
We denote by $\R_+$ (respectively,  $\R_-$) the set of non-negative (respectively, non-positive) real numbers.
Given an $n \times n$ matrix $Z$ and index sets $I , J \subseteq \{1,\ldots,n\}$, denote by $Z_{IJ}$ the submatrix of $Z$ obtained
by removing all the rows of $Z$ not in $I$ and all the columns of $Z$ not in $J$ . The matrix $Z_I$ is
the submatrix of $Z$ with columns specified by $I$. Particularly, $Z_i$ is the $i$-th column
of $Z$ and $Z_{ij}$ is the entry of $Z$ at $(i, j)$ position. Denote by $Z^{\dagger}$ the Moore–Penrose generalized inverse of $Z$.  
Finally, the cardinality of the set $I\subset \N$, where $\N$ stands for the set of natural numbers,  is denoted by $|I|$.

Throughout this paper, we denote by   $\R^{n\times m}$  the space of all real $n\times m$   matrices and by $\S^n$   the space of all real $n\times n$ symmetric  matrices
equipped with the inner product 
$$
\la X, Y\ra=\tr(XY),\quad X,Y\in \S^n.
$$
The induced Frobenius norm of $X\in \S^n$ is defined via the trace inner product by $\|X\|=\sqrt{\tr(X^2)}$.
Given $X \in \S^n$, its eigenvalues, in   nonincreasing order, are denoted by
$$
 \lambda_{1} (X) \geq \lambda_{2} (X) \geq \cdots \geq \lambda_{n} (X).
 $$  
 For any vector $x=(x_1,\ldots,x_n) \in \R^n$,  denote by $\Diag (x)$ the diagonal matrix whose   $i$-th diagonal entry is   $x_i$ for any $i=1,\ldots,n$.
 The set of all real $n \times n$ orthogonal matrices is denoted by $\O^n$. It is known that  for any $X \in \S^n$, there exists an orthogonal matrix $U$ for which  we have 
\begin{equation}\label{specdocom}
X  =  U \Diag (\lambda (X)) U^\top\quad \mbox{with}\;\;\lm(X):=\big(\lm_1(X),\ldots,\lm_n(X)\big).
\end{equation}
For a given matrix $X\in \S^n$, the set of  such  orthogonal  matrices $U$ is denoted by $\O^n (X)$. We say that two matrices $X,Y\in \S^n$ admit   a simultaneous spectral decomposition if there exists 
$U\in \O^n$ such that $U^\top XU$ and $U^\top YU$ are diagonal matrices. The   matrices  $X$ and $Y$ are said to have a simultaneous {\em ordered} spectral decomposition if there exists 
$U\in \O^n$ such that $U^\top XU=\Diag (\lambda (X)) $ and $U^\top YU=\Diag (\lambda (Y)) $. It is well-known that for any two matrices $X,Y\in \S^n$, the estimate 
\begin{equation}\label{lipei}
\|\lm(X)-\lm(Y)\|\le \|X-Y\|
\end{equation}
always holds. Moreover, equality in this estimate amounts to $X$ and $Y$ admitting a  simultaneous  ordered spectral decomposition.
It is not hard to see that the estimate in \eqref{lipei} amounts to the trace inequality, known as Fan's inequality, 
\begin{equation}\label{fani}
\la X,Y\ra\le \la \lm(X),\lm(Y)\ra,\quad X,Y\in \S^n.
\end{equation}
Assume that  $\mu_{1}(X) >  \cdots> \mu_{r}(X) $ are distinct eigenvalues of $X\in \S^n$ and   define  then  the index sets
\begin{equation}\label{index}
\alpha_{m} :=\big\{  i \in  \{ 1,\ldots, n\}   | \; \lambda_{i} (X)  = \mu_{m}(X)\big\}\quad \mbox{for all}\;\; m=1,\ldots,r.
\end{equation}
Moreover, define $\l_i (X)$ for any $i \in \{1, \ldots, n\} $ to be the number of eigenvalues of $X$ that are equal to $\lambda_{i} (X)$
but are ranked before $\lm_i(X)$ including $\lambda_{i} (X)$. This integer allows us to locate $\lm_i(X)$
in the group of the eigenvalues of  $X$ as follows:
\begin{equation}\label{ell}
\lm_1(X)\ge \cdots \ge \lm_{i-\ell_i(X)}>\lm_{i-\ell_i(X)+1}(X)= \cdots =\lm_i(X)\ge \cdots\ge \lm_n(X).
\end{equation}
 Note that  the index sets $\al_m$ present a partition of $\{1,\ldots,n\}$, meaning that $\{1,\ldots,n\}=\cup_{m=1}^r\al_m$.
 In what follows,   we often drop $X$ from  $\ell_i (X)$ when the dependence of $\l_i$ on $X $ can be seen clearly from the context. 
 Given an $n\times n$ matrix $W$, it is not hard to see that for any $U\in \O^n$ and any $m=1,\ldots,r$, we always have 
 \begin{equation}\label{dfa}
 U_{\al_m}^\top U WU^\top U_{\al_m}=W_{\al_m\al_m}.
 \end{equation}
 This simple observation will often be utilized in Section~\ref{sect05}. 
 
The following estimates are an easy consequence of \cite[Proposition~1.4]{t1} (cf. see    the proof of \cite[Theorem~1.5]{t1}) and    play a major role in  our second-order variational analysis of eigenvalue functions in this paper.
\begin{Proposition}[first-order expansion of eigenvalue functions]\ \label{olemma}
Assume that  $X \in \S^n$ has the eigenvalue decomposition \eqref{specdocom} for some $U \in \O^n (X)$.  Let 
  $\mu_{1} >  \cdots> \mu_{r} $ be distinct eigenvalues of $X$.  Then for any  $H\in \S^n$ that $H\to 0$ and any $i\in \{1,\ldots,n\}$ the estimates 

\begin{equation}\label{secondexp1}
\lambda_{i} (X +  H) = \lambda_{i} (X ) +  \lambda_{\l_i} \big( U_{\al_m}^\top  H U_{\al_m}  + U_{\al_m}^\top  H ( \mu_m I-X)^{\dagger} H U_{\al_m}  \big) + O(\|  H \|^3)
\end{equation}
and 
\begin{equation}\label{fexpan}
\lambda_{i} \big( X + H  \big) = \lambda_{i} (X)  +  \lambda_{\l_i} (U_{\al_m}^{\top} H U_{\al_m}) +  O (\| H \|^2 )   
\end{equation}
hold, where    $m\in \{1,\ldots,r\}$ with  $i \in \alpha_m$. 
\end{Proposition}
 
 Note that the estimate \eqref{fexpan} clearly tells us that the eigenvalue function $\lm_i(\cdot)$, $i\in \{1,\ldots,n\}$, is directionally differentiable at $X$ at any direction $H\in \S^n$ and 
 its directional derivative $\lambda_i' ( X ; H ) $ can be calculated by 
\begin{equation}\label{dirla}
\lambda_i' ( X ; H ) := \lim_{t \dn 0} \dfrac{ \lambda_i (X + t H ) - \lambda_i (X ) }{t}=  \lambda_{\l_i} (U_{\al_m}^{\top} H U_{\al_m})  
\end{equation}
 where    $m\in \{1,\ldots,r\}$ with  $i \in \alpha_m$.  In another words, we have 
 $$
 \lambda^{'} ( X ; H )=\big(\lambda (U_{\al_1}^{\top} H U_{\al_1}),\ldots,\lambda (U_{\al_r}^{\top} H U_{\al_r}) \big) 
 $$
where  $\lambda (U_{\al_m}^{\top} H U_{\al_m})\in  \R^{|\al_m|}$ for any  $m = 1,\ldots,r$.
 This observation indicates that both estimates in \eqref{secondexp1} and \eqref{fexpan} are, indeed, 
 a first-order estimate of eigenvalue function $\lm_i(\cdot)$. To obtain a second-order estimate, we need to repeat a similar argument for each of the symmetric matrices $U_{\al_m}^{\top} H U_{\al_m}$
 for any $m\in \{1,\ldots,r\}$. To this end, fix $m\in \{1,\ldots,r\}$ and observe that  $U_{\al_m}^{\top} H U_{\al_m}\in \S^{| \al_m |}$. Thus,  we find $Q_m   \in \O^{| \al_m |} (U_{\al_m}^\top  H U_{\al_m}) $
 such that 
 \begin{equation}\label{sdt1}
 U_{\al_m}^{\top} H U_{\al_m}=Q_m\Lambda(U_{\al_m}^\top  H U_{\al_m})Q_m^\top.
 \end{equation}
Denote by $\eta_{1}^{m} >     \cdots> \eta_{\rho_m}^{m}$ the distinct eigenvalues of $U_{\al_m}^\top H U_{\al_m}$.
 Similar to \eqref{index},  define the index sets 
\begin{equation}\label{index2}
 \beta_{j}^m :=\big \{ i \in \{1,\ldots,  | \al_m| \} \big |\; \lambda_i (U_{\al_m}^\top H U_{\al_m}) = \eta_{j}^m  \big\} \quad \mbox{for all}\;\;  j=1,\ldots, \rho_{m}.
\end{equation}

To state the promised second-order estimate for eigenvalue functions, we need to clarify some of indices,  appeared therein. To do so, pick $i\in \{ 1,\ldots,n  \}$, and
observe that there is $m\in \{1,\dots,r\}$ such that $i\in \al_m$ and  that $\l_i(X)\in  \{1,\ldots, |\al_m| \}$, where $\l_i(X)$ is defined by   \eqref{ell}. 
Furthermore, we find $j \in \{1,\ldots, \rho_m \}$ such that $\l_i(X) \in \beta^{m}_{j}$. Define now the integer $\l'_i (X,H)$ by 
$$
\l'_i (X,H)=\l_{\l_i(X)}(U_{\al_m}^{\top} H U_{\al_m}),
$$
which , in fact, signifies the number of eigenvalues of $U_{\al_m}^\top H U_{\al_m}$ that are equal to $\lambda_{\l_i(X)} (U_{\al_m}^\top H U_{\al_m})$ 
but are ranked before $\lambda_{\l_i(X)} (U_{\al_m}^\top H U_{\al_m})$ including $\lambda_{\l_i(X)} (U_{\al_m}^\top HU_{\al_m})$.
As before, we often drop $X$ and $H$ from  $\l'_i (X,H)$ when the dependence of $\l'_i$ on $X $ and $H$ can be seen clearly from the context.   
 In summary, for any
$i\in \{ 1,\ldots,n  \}$, there are $m\in \{1,\dots,r\}$ and $j \in \{1,\dots, \rho_m \}$ for which we have, respectively,
\begin{equation}\label{dex2}
i \in \al_m\quad \mbox{and}\quad \l_i(X) \in \beta^{m}_{j}.
\end{equation}
The following second-order estimate of eigenvalue functions was established in \cite[Proposition~2.2]{t1} and has important consequences for second-order variational analysis 
of eigenvalue functions; see also \cite[Proposition 2.1]{zzx}. 
\begin{Proposition}[second-order expansion of eigenvalue functions]\label{secondexp}
Assume that  $X \in \S^n$ has the eigenvalue decomposition \eqref{specdocom} for some $U \in \O^n (X)$ and that $H,W\in \S^n$. 
 Let 
  $\mu_{1} >  \cdots> \mu_{r} $ be distinct eigenvalues of $X$. 
Then for any $t >0$ sufficiently small and any $i \in \{ 1,\ldots,n  \}$  we have 
\begin{equation*} 
\lambda_{i} (Y(t))= \lambda_{i} (X ) + t \lambda_{\l_i} ( U_{\al_m}^\top  H U_{\al_m} )+ \sm t^2  \lambda_{\ss \l'_i} \Big( {R_{mj}}^{\top}  \big( U_{\al_m}^\top ( W + 2  H (\mu_mI  - X)^{\dagger} H ) U_{\al_m} \big) R_{mj} \Big) + o(t^2),
\end{equation*}
where $Y(t):=X + t H + \frac{1}{2} t^2 W+o(t^2)\in \S^n$, and  $ R_{mj}:=(Q_m)_ {\beta^m_j}$ with $Q_m$ and $\beta^m_j$ taken from \eqref{sdt1} and \eqref{dex2}, respectively, and where the indices  $m$ and $j$ come from \eqref{dex2}. 
\end{Proposition}
 
In the framework of  Proposition~\ref{secondexp}, we can  conclude from \eqref{dirla} that for any $i \in \{ 1,\ldots,n \}$, 
the parabolic second-order directional derivative of the eigenvalue function $\lm_i(\cdot)$ at $X$ for $H$ with respect to $W$, denoted   $\lambda_i^{''} (X; H, W)$, exists. 
Recall that the latter concept is defined by 
\begin{equation*} 
\lambda_i^{''} (X; H, W)= \lim_{ 
   t\dn 0} \dfrac{ \lambda_i (X + t H + \frac{1}{2} t^2 W) - \lambda_i (X ) - t \lambda'_i ( X ; H ) }{\frac{1}{2}t^2}.
\end{equation*}
According to Proposition~\ref{secondexp}, we can conclude further that 
\begin{equation}\label{paralam}
\lambda_i^{''} (X; H, W)= \lambda_{\ss \l'_i} \Big( {R_{mj}}^{\top}  \big( U_{\al_m}^\top ( W + 2  H (\mu_m I  - X)^{\dagger} H ) U_{\al_m} \big) R_{mj} \Big).
\end{equation}
Combining these with \eqref{dirla} brings us to the following estimate for the eigenvalue function $\lm(\cdot)$ from \eqref{specdocom}, important for our development in this paper.  
 
\begin{Corollary}
Assume that  $X \in \S^n$ has the eigenvalue decomposition \eqref{specdocom} for some $U \in \O^n (X)$ and that $H,W\in \S^n$. Then for any $t >0$ sufficiently small
we have 
\begin{equation}\label{secondexp2}
\lambda \big(X + t H + \sm t^2 W+o(t^2)\big)= \lambda (X ) + t \lambda'( X;H)+ \sm t^2 \lambda^{''} (X; H, W)+ o(t^2).
\end{equation}
 \end{Corollary}
 
We proceed with recalling some concepts, utilized extensively in this paper. Given a nonempty set $C\subset\X$ with $\ox\in C$, the  tangent cone $T_ C(\ox)$ to $C$ at $\ox$ is defined by
\begin{equation*}\label{2.5}
T_C(\ox)=\big\{w\in\X|\;\exists\,t_k{\dn}0,\;\;w_k\to w\;\;\mbox{ as }\;k\to\infty\;\;\mbox{with}\;\;\ox+t_kw_k\in C\big\}.
\end{equation*}
We say a tangent vector $w\in T_C(\ox)$ is {\em derivable} if there exist a constant  $\ve>0$ and an arc $\xi:[0,\ve]\to C$ such that  $\xi(0)=\ox$ and $\xi'_+(0)=w$, where  $
\xi'_+(0):=\lim_{t\dn 0} {[\xi(t)-\xi(0)]}/{t}
$
signifies the right derivative of $\xi$
at $0$.
The set $C$ is called geometrically derivable at $\ox$ if every tangent vector $w$ to $C$ at $\ox$ is derivable. The geometric derivability of $C$ at $\ox$ can be equivalently described by the sets $[C-\ox]/{t}$ 
converging to $T_C(\ox)$  as $t\dn 0$ in the sense of the Painlev\'e-Kuratowski set convergence (cf. \cite[Definition~4.1]{rw}). 

 Given a function $f:\X \to \oR$, its domain is defined  by 
$\dom f =\big\{ x \in \X|\; f(x) < \infty \big \}$.
The  function $f$ is called locally Lipschitz continuous around $\ox $ {\em relative} to $C \subset \dom f$
with constant $\ell \ge0 $
if  $\ox \in C$ with $f(\ox)$ finite  and there exists  a neighborhood $U$ of $\ox$ such that 
\begin{equation*} \label{lipwrtdomain}
 |f(x)  - f(y ) | \leq  \ell\, \| x - y \|  \quad \mbox{for all }\;       x , y \in U \cap C.    
\end{equation*}
Such a function is called locally Lipschitz continuous relative to $C$ if  it is locally Lipschitz continuous around every $\ox\in C $ relative to $C$. 
Piecewise linear-quadratic functions (not necessarily convex) and an indicator function of a nonempty set are important examples of 
functions that are locally Lipschitz continuous relative to their {domains}.    
The subderivative function of $f$ at $\ox$, denoted by $\d f(\ox)\colon\X\to \oR$,  is defined by
$$
{\mathrm d}f(\ox)(\ow)=\liminf_{\substack{
   t\dn 0 \\
  w\to \ow
  }} {\frac{f(\ox+tw)-f(\ox)}{t}}.
$$
 When $f$ is convex, its  subdifferential    at $\ox$ with $f(\ox)$ finite, denoted by $\sub f(\ox)$,  is understood in the sense of convex analysis, namely $v\in \sub f(\ox)$ if   $f(x)\ge  f(\ox)+\la v,x-\ox\ra $ for any $x\in \X$.
 Given a nonempty convex set $C\subset \X$, its normal cones to $C$ at $\ox\in C$ is defined  by $ N_C(\ox)= \sub  \dd_C(\ox) $.
 The second-order tangent set to $C\subset \X$ at $\ox\in C$ for a tangent vector $w\in T_C(\ox)$ is given by
\begin{equation}\label{2tan}
T_C^2(\ox, w)=\big\{u\in\X|\;\exists\,t_k{\downarrow}0,\;\;u_k\to u\;\;\mbox{ as }\;k\to\infty\;\;\mbox{with}\;\;\ox+t_kw+\sm t_k^2 u_k\in C\big\}.
\end{equation}
A set $C$ is called {parabolically derivable} at $\ox$ for $w$ if $T_C^2(\ox, w)$ is nonempty and for each $u\in T_C^2(\ox, w)$ there are $\ve>0$ and an arc $\xi:[0,\ve]\to C$ with $\xi(0)=\ox$, 
$\xi'_+(0)=w$, and $\xi''_+(0)=u$, where $\xi''_+(0):=\lim_{t\dn 0}[{\xi(t)-\xi(0)-t\xi'_+(0)}]/{\sm t^2}$.
It is   known that if $C$ is convex and parabolically  derivable at $\ox$ for  $w$, then  the second-order tangent set $T_C^2(\ox, w)$
is a nonempty convex set in $\X$ (cf. \cite[page~163]{bs}). Below, we record a simple characterization of parabolic derivability of a set, used extensively in our paper.
 \begin{Proposition} \label{pdch} Assume that $C\subset \X$, $\ox\in C$,  and $w\in T_C(\ox)$. Then, the following are equivalent:
\begin{itemize}[noitemsep,topsep=0pt]
 \item [{\rm (a)}]  $C$ is parabolically derivable at $\ox$ for $w$;
 \item [{\rm (b)}]  for any $u\in T_C^2(\ox,w)$, we find $\ve>0$ such that 
 $$
 \ox+ tw+\sm t^2 u+o(t^2)\in C\quad \mbox{for all}\;\; t\in [0,\ve].
 $$
 \end{itemize}
\end{Proposition}
 \begin{proof} If  (b) is satisfied, one can define $\xi(t)=\ox+ tw+\sm t^2 u+o(t^2)$ for any $ t\in [0,\ve]$ with $\ve$ taken from (b).
 It is easy to see that  $\xi(0)=\ox$, 
$\xi'_+(0)=w$, and $\xi''_+(0)=u$, which confirm  (a). Suppose that  (a) holds and then pick $u\in T_C^2(\ox, w)$. Since   $C$ is parabolic drivable at $\ox$ for $w$, we find 
 $\ve>0$ and an arc $\xi:[0,\ve]\to C$  such that $\xi(0)=\ox$, $\xi_+'(0)=w $, and 
$\xi''_+(0)=u$. Set $u(t)=(\xi(t)-\xi(0)- t\xi_+'(0))/{\sm t^2}$ for any $t\in [0,\ve]$. It follows from $\xi''_+(0)=u$ that $u(t)\to u$
as $t\dn 0$. By the definition of $\xi''_+(0)$, we get 
$$
 \xi(0)+ t\xi_+'(0)+\sm t^2 u(t)=\xi(t)\in C.
$$
 One other hand, one can express $\xi(t)$ equivalently as 
$$
\xi(t)= \xi(0)+ t\xi_+'(0)+\sm t^2 u+ v(t)\quad \mbox{with}\;\; v(t):=\sm t^2(u(t)- u).
$$
Clearly, we have $v(t)=o(t^2)$, which proves (b) and hence completes the proof.
 \end{proof}

\section{Subderivatives of Spectral Functions}\sce \label{sect03}
In this section, we present two important results about the spectral functions, central to our developments in this paper. 
 The first one presents a counterpart of the estimate in \eqref{specmsqc} for symmetric functions in Proposition~\ref{mscqforset}.
The second one presents   a chain rule for the subderivative of spectral functions in Theorem~\ref{specsubd}.
To state the former about symmetric functions, recall that a function $\th:\R^n\to \oR$ is called symmetric if for every 
$x\in \R^n$ and every $n\times n$ permutation matrix $Q$, we have $\th(Qx)=\th(x)$. Recall also that $Q$ is a permutation matrix if all its components are either $0$ and 
$1$ and each row and each column has exactly one nonzero element. We denote by $\P^n$ the set of all $n\times n$ permutation matrices.  As pointed out before, for any spectral function 
$g:\S^n\to \oR$, there exists a symmetric function  $\th:\R^n\to \oR$ satisfying \eqref{spec}. Indeed, $\th$ can be chosen as 
the restriction of $g$ to   diagonal matrices, namely 
\begin{equation}\label{spec2}
\th(x)=g\big(\Diag(x)\big) \quad \mbox{for all}\;\; x\in \R^n. 
\end{equation}
A set  $C\subset \S^n$ is called a spectral set if   $\dd_C$ is a spectral function. 
Likewise, $\Th\subset \R^n$ is called a symmetric set if $\dd_\Th$ is a symmetric function. 
Similar to \eqref{spec}, it is easy to see that for any spectral set $C \subset \S^n$,  there exists a symmetric set $\Theta \subset \R^n$ such that
\begin{equation}\label{spectralset}
C = \big\{X \in \S^n  \big| \;  \lambda (X) \in \Theta \big\},
\end{equation}
where $\Th$ can be chosen as 
\begin{equation}\label{setth}
\Th=\big\{ x \in \R^n \big| \; \Diag (x) \in C \big\}.
\end{equation}
The composite forms  \eqref{spec} and  \eqref{spec2}  readily imply, respectively,  that 
\begin{equation}\label{domcs}
\dom g = \big\{ X \in \S^n \big| \;  \lambda (X) \in \dom \theta   \big\}\quad \mbox{and}\quad \dom \th =\big\{ x \in \R^n \big| \; \Diag (x) \in \dom g \big\}.
\end{equation}
Next, we are going to justify a similar estimate as \eqref{specmsqc} for domains of symmetric functions, which allows us to show via the established theory for composite functions in \cite{ms} that 
second-order variational properties of spectral functions are inherited by symmetric functions. 
\begin{Proposition}\label{mscqforset} Let $g:\S^n\to \oR$  be a spectral  function, represented by \eqref{spec}. Then  for any  $x \in \R^n$  we have 
\begin{equation}\label{revset2}
\dist (x , \dom \th) = \dist \big(\Diag (x) ,\dom g\big),
\end{equation}
where $\th$ is taken by \eqref{spec}.
\end{Proposition} 
\begin{proof} For any  $x \in \R^n $, we know that  there exist   a permutation matrix $P \in \O^n$ such that $ \lambda (\Diag(x)) = Px$. 
Since $\th$ is a symmetric function, $\dom \th$ is a symmetric set. Thus for any $X \in \dom g$,  we have $P^\top \lambda (X) \in \dom \th.$ This, coupled  with \eqref{lipei}, leads us to
\begin{eqnarray*}
\| \Diag (x) - X\| \geq \| \lambda(\Diag(x)) - \lambda(X) \| &=& \| Px - \lambda (X)  \| \\\nonumber
&=& \| x - P^\top \lambda (X)  \| \geq \dist (x , \dom \th)
\end{eqnarray*}
for all $X\in \dom g$, which in turn   brings us to 
$$
\dist (x , \dom \th) \leq \dist \big(\Diag (x) ,\dom g\big).
$$
To prove the opposite inequality, pick any $y \in \dom \th.$  By \eqref{domcs}, we get $\Diag (y) \in \dom g$, which implies that  
\begin{equation*}
\| x - y \| = \| \Diag (x) - \Diag(y)  \| \geq \dist(\Diag (x) , \dom g).
\end{equation*}
Combining these clearly justifies \eqref{revset2}.
\end{proof}

Note that the identity in \eqref{specmsqc} allows us to show that second-order variational properties of a symmetric function $\th$ from \eqref{spec} are disseminated to the spectral function $g$. 
Appealing to \eqref{revset2}, we will show in the coming sections that those variational properties of the spectral function $g$ are inherited by the symmetric function $\th$ from \eqref{spec}.
This will be achieved using   the second-order variational  theory in \cite{ms} for the composite form \eqref{spec2}.
Note also that the results in \cite{ms} were proven under a constraint qualification, which is  similar to \eqref{mscq}. According to \eqref{revset2}, such a constraint qualification 
automatically holds for the composite form \eqref{spec2}. Moreover,  the  inner mapping $x \mapsto \Diag(x)$ in this composite form is   twice continuously differentiable, which 
allows us to exploit the results in \cite{mms1, ms}.  

\begin{Proposition} \label{chss}
Let $g:\S^n\to \oR$  be a spectral  function, represented by \eqref{spec}, and let the symmetric function $\th$, taken from \eqref{spec}, be locally Lipschitz continuous   relative to its domain.
Then for any $X\in \S^n$ with $g(X)$  finite and   any $v\in \R^n$, we have
\begin{equation}\label{dchain3}
\d \th(\lm( X )) (v) =  \d g \big(\Diag(\lambda(X)) \big)(\Diag (v)).
\end{equation}  
In particular, if $g=\dd_C$, where $C\subset \S^n$ is a spectral set, then  we get for any $X\in C$ that  
\begin{equation}\label{dchain4}
T_\Th(\lm( X )) = \big\{ v\in \R^n|\; \Diag (v)\in T_{C}(\Diag(\lm(X))\big\},
\end{equation}  
where $\Th$ is taken from \eqref{spectralset}.
\end{Proposition}
\begin{proof} It follows from  \eqref{spec} that the symmetric function $\th$ satisfies \eqref{spec2}, which means that 
$\th$ can be represented as a composite function of $g$ and the linear mapping $x\mapsto \Diag(x)$ with $x\in \R^n$. We also deduce from the imposed assumption on  $\th$ and the inequality in \eqref{lipei}
that $g$ is locally Lipschitz continuous relative to its domain. This, together with \eqref{revset2} and \cite[Theorem~3.4]{mms1}, justifies \eqref{dchain3}.
To justify \eqref{dchain4}, recall the representation $\Th$ from \eqref{setth}, which can be equivalently expressed as $\dd_\Th(x)=\dd_C(\Diag(x))$  for any $x\in \R^n$.
The claim equality in \eqref{dchain4} results from \eqref{dchain3} and  the fact that  $\d \th(\lm( X )) =\dd_{T_\Th(\lm( X )) }$ and $ \d g \big(\Diag(\lambda(X)) \big)=\dd_{T_C(\ss{\Diag}(\lm(X)))}$.
\end{proof}

\begin{Remark}[symmetric property of subderivatives]\label{sysub}{\rm If $\th:\R^n\to \oR$ is a symmetric function and $X\in \S^n$ with $\th(\lm(X))$ finite, one may wonder whether the subderivative $\d\th(\lm(X))$ is a symmetric function.
This can be easily disproven by taking $\th=\dd_{\R_-^n}$ and $X\in \S^n$ such that all its eigenvalues are not the same; see Example~\ref{tans} for more details.
While this may seem disappointing, we can show that $\d\th(\lm(X))$ is a symmetric function with respect to a subset of $\P^n$. Indeed, assume that $\P^n_X$ is the set of 
all $n\times n$ block diagonal matrices in the form $Q=\Diag(P_1,\ldots,P_r)$, where $P_m\in \R^{|\al_m|\times |\al_m|}$ is a  permutation matrix for any $m=1,\ldots,r$ with $\al_m$ taken from \eqref{index} and $r$
being the number of distinct eigenvalues of $X$. It is clear that $\P^n_X\subset \P^n$ and that  if $Q\in \P^n_X $, then we have $Q\lm(X)=\lm(X)$. Moreover, for any $v\in \R^n$ and $Q\in \P^n_X$, we get 
\begin{eqnarray*}
\d \th(\lm( X )) (v) &=& \liminf_{\substack{t\dn 0\\v'\to v}}\frac{\th \big(\lambda (X)+tv'\big)-\th \big(\lambda (X)\big)}{t}\\
&=& \liminf_{\substack{t\dn 0\\ v'\to v}}\frac{\th \big(\lambda (X)+tQv'\big)-\th \big(\lambda (X)\big)}{t}\\
&\ge & \liminf_{\substack{t\dn 0\\ w\to Qv}}\frac{\th \big(\lambda (X)+tw\big)-\th \big(\lambda (X)\big)}{t}=\d \th(\lm( X )) (Qv).
\end{eqnarray*}
Since $Q^{-1}=\Diag(P_1^{-1},\ldots,P_r^{-1})$, we can show similarly that $\d \th(\lm( X )) (v)\le \d \th(\lm( X )) (Qv)$ for any $v\in \R^n$ and $Q\in \P^n_X$, which leads us to 
$$
\d \th(\lm( X )) (v)= \d \th(\lm( X )) (Qv)\quad \mbox{for all}\;\; v\in \R^n, \; Q\in \P^n_X,
$$
demonstrating that $\d \th(\lm( X )) $ is a symmetric function with respect to $\P^n_X$.
}
\end{Remark}

We proceed by proving a chain rule for subderivatives of spectral functions. We begin with recalling a useful characterization of the subdifferential of the spectral functions.
\begin{Proposition} \label{subsp} Assume that  $\th:\R^n\to \oR$ is a proper,  lower semicontinuous {\rm(}lsc{\rm )}, convex, and  symmetric  function. Then the following properties are equivalent:
\begin{itemize}[noitemsep,topsep=0pt]
 \item [{\rm (a)}]  $Y\in \sub (\th\circ\lm)(X)$;
 \item [{\rm (b)}]  $\lm(Y)\in \sub \th(\lm(X))$ and the matrices $X$ and $Y$ have simultaneous ordered spectral decomposition, meaning that 
 there exists $U\in \O^n(X)\,\cap\, \O^n(Y)$ such that 
$$
X=U\Lm(X)U^\top\quad \mbox{and}\quad Y=U\Lm(Y)U^\top,
$$
where $\Lm(X)=\Diag(\lm(X)\big)$ and $\Lm(Y)=\Diag(\lm(Y)\big)$.
\end{itemize}
\end{Proposition} 

\begin{proof} It follows from \cite[Corollary~5.2.3]{bl} that $\th\circ \lm$ is lsc and convex if and only if $\th$ is lsc and convex.
  The claimed equivalence then results from  \cite[Theorem~5.2.4]{bl}.
\end{proof}
 
 Given a matrix $X\in \S^n$ with $r$ distinct eigenvalues and the index sets $\al_m$, $m=1,\ldots,r$,  from \eqref{index}, recall that 
$\cup_{m=1}^r\al_m=\{1,\ldots,n\}$. In what follows,  we  partition a vector $p\in \R^n$ into  $(p_{\al_1},\ldots, p_{\al_r})$, where $p_{\al_m}\in \R^{|\al_m|}$ for any $m=1,\ldots,r$.
 
 \begin{Theorem}[subderivatives of spectral functions]\label{specsubd}
Let $\th:\R^n\to \oR$ be a symmetric function  and let $X\in \S^n$ with $(\th\circ \lm)(X)$   finite. If $\th$ is either   lsc and convex  with $\sub \th(\lm(X))\neq \emptyset$ or locally Lipschitz continuous around $\lambda(X)$ relative to its domain, then for all $H \in \S^n$ we have
\begin{equation}\label{dchain}
\d (\th\circ\lm)( X ) (H) =  \d \theta (\lambda(X)) (\lambda^{'} (X ; H)).
\end{equation}   
\end{Theorem}  
\begin{proof}
Pick any $ H \in\S^n $ and deduce from Proposition~\ref{olemma} that $\lambda' (X ; .)$ is a Lipschitz continuous and positively homogeneous function. Moreover, we have $\lambda' (X ; E)+ {O(t^2 \|E\|^2 )}/{t}\to\lambda' (X ; H)$ as $t\dn 0$ and $E \to H$.  
This and the definition of subderivative give us  the relationships
\begin{eqnarray}
\disp\d(\th \circ \lambda)(X)(H)&=&\liminf_{\substack{t\dn 0\\E\to H}}\frac{\th \big(\lambda (X+t E)\big)-\th \big(\lambda (X)\big)}{t}\nonumber\\
&=&\liminf_{\substack{t\dn 0\\ E\to H}}\frac{\th \big( \lambda (X)+t \lambda'(X ; E) +O(t^2 \|E\|^2 )\big)-\th \big(\lambda (X) \big)}{t}\nonumber\\
&=&\liminf_{\substack{t\dn 0\\ E \to H}}\frac{\th \big(\lambda (X)+t[\lambda'(X ; E)+ {O(t^2 \|E\|^2 )}/{t}]\big)-\th \big(\lambda (X) \big)}{t}\nonumber\\
&\ge&\d\th \big(\lambda (X)\big)\big(\lambda'(X ; H)\big)\label{blm},
\disp
\end{eqnarray}
which verify the inequality ``$\ge$" in \eqref{dchain}. To justify  the opposite inequality in \eqref{blm}, observe that if $\d \th (\lambda (X))(\lambda'(X ; H)) =\infty$, 
the latter inequality clearly holds. Thus,  assume that   $\d \th (\lambda (X))(\lambda'(X ; H))<\infty$. If $\th$ is lsc and convex,   $\th\circ \lm$ is lsc and convex due to \cite[Corollary~5.2.3]{bl}.
Moreover, it follows from \cite[Theorem~6]{l99} and $\sub \th(\lm(X))\neq \emptyset$  that $\sub (\th \circ \lambda)(X)\neq \emptyset$. 
 Thus, it follows  from \cite[Theorem~8.30]{rw}   
that $\d (\th \circ \lambda)(X) (H) = \sup_{Y \in \sub (\th \circ \lambda)(X) }   \la Y, H \ra $.  Let $\varepsilon > 0$    and  choose $Y \in \sub (\th \circ \lambda)(X)$ such that 
$$    
 \d (\th \circ \lambda)(X) (H) \leq  \ve  + \la Y, H \ra .  
$$
Since $Y \in \sub (\th \circ \lambda)(X)$, it follows from Proposition~\ref{subsp} that there is $U\in \O^n(X)\,\cap\, \O^n(Y)$ such that $\lm(Y)\in \sub \th(\lm(X))$.  
Set   $\Lm(Y):=U^\top YU$ and use Fan's inequality to conclude  
\begin{eqnarray*}
	\d (\th \circ \lambda)(X) (H) \leq  \ve  + \la Y, H \ra  =   \ve  + \la \Lm(Y), U^T H U \ra  &= &\ve  +  \sum_{m=1}^{r} \la \Lm(Y)_{\al_m\al_m}  ,  U_{\al_m}^\top  H U_{\al_m} \ra\\
	&\leq&  \ve  +  \sum_{m=1}^{r} \la \lambda (Y)_{\al_m}, \lambda ( U_{\al_m}^\top H U_{\al_m}) \ra    \\
	&=&   \ve  +  \la \lambda (Y )  ,  \lambda^{'} (X; H) \ra   \\
	&\leq&  \ve  +  \d \th ( \lambda(X)) (\lambda^{'} (X; H) ),
\end{eqnarray*}
where the last inequality results from the fact that $\th$ is convex and  $\lm(Y)\in \sub \th(\lm(X))$. 
Letting $\varepsilon \dn 0$, we get the opposite inequality in \eqref{blm}, which proves \eqref{dchain} in this case.
Suppose now that $\th$ is  locally Lipschitz continuous around $\lambda(X)$ relative to its domain.
To prove the opposite inequality in \eqref{blm}, by definition,  there exist sequences $t_k\dn 0$ and $v_k\to \lambda'(X ; H)$ such that
\begin{equation}\label{subd0}
\d \th (\lambda (X))(\lambda'(X ; H)) =\lim_{\substack{k\to\infty}}\frac{\th \big(\lambda(X)+t_k v_k\big)-\th \big(\lambda (X) \big)}{t_k}.
\end{equation}
Since  $\d \th (\lambda (X))(\lambda'(X ; H))<\infty$, we can assume without loss of generality  that    $\lambda (X) +t_k v_k \in\dom \th $ for all $k\in\N$.
Take the function $g$ from \eqref{spec} and appeal to   \eqref{specmsqc} to get    
\begin{equation*}
{\rm dist}(X + t_k H, \dom g ) =\dist \big(\lambda (X + t_k H),\dom\th \big),\quad k\in\N,
\end{equation*}
which in turn brings us to the relationships
\begin{eqnarray*}
{\rm dist}\Big(H, \frac{\dom g - X}{t_k}\Big)&= &\frac{1}{t_k}\, \dist \big( \lambda (X) + t_k  \lambda' (X, H) +O(t^{2}_k), \dom\th \big)\nonumber \\
&\le &\frac{1}{t_k}\,\big \| \lambda (X) +t_k \lambda' (X,H) +O(t^{2}_k) - \lambda (X) -t_k  v_k\big\|\nonumber\\
&=& \big\| \lambda' (X ; H) -v_k+\frac{O(t^{2}_k)}{t_k}\big\|\;\mbox{ for all }\;k\in\N.
\disp
\end{eqnarray*}
So for each $k\in \N$, we find a matrix $E_k\in \S^n$ such that   $X+t_kE_k\in \dom g $ and 
$$
\|H - E_k\| < \big\| \lambda' (X; H) - v_k+\frac{O(t^{2}_k)}{t_k} \big\|+\frac{1}{k},
$$
which in turn yields    $E_k \to H$ as $k\to\infty$. Combining these with \eqref{subd0} and \eqref{fexpan}, we arrive at
\begin{eqnarray*}
\d \th \big( \lambda (X ) \big)\big(\lambda' (X ; H)\big)&=& \disp\lim_{k\to\infty}\Big[\frac{ g(X + t_k E_k )-g (X)}{t_k}+\frac{\th \big( \lambda (X) +t_k v_k\big)-\th \big(\lambda (X +t_k E_k)\big)}{t_k}\Big]\\
&\ge & \disp\liminf_{k\to\infty}\frac{g  (X  +t_k E_k)-g(X)}{t_k}-\kappa \lim_{k\to\infty}\big\|\frac{\lambda (X + t_k E_k)- \lambda (X)}{t_k}-v_k\big\|\\
&\ge&\d g(X)(H)-\kappa \lim_{k\to\infty}\big\| \lambda' (X ; E_k )+\frac{O(t^{2}_k)}{t_k}-v_k\big\|=\d g (X)(H),
\end{eqnarray*}
where $\kappa\ge 0$ is a Lipschitz constant of $\th$ around $\lambda(X)$ relative to its domain. This verifies the inequality ``$\le$" in \eqref{dchain} and completes the proof of the theorem.
\end{proof}

As an immediate conclusion of Theorem~\ref{specsubd}, we obtain a simple representation of tangent cones to spectral sets.
\begin{Corollary}[tangent cone to the spectral sets]\label{tansp} Let $C$ be a   spectral set represented by \eqref{spectralset}. Then for any  $X \in C$, we have 
\begin{equation*}\label{tspec}
T_C(X) = \big\{ H \in \S^n \big|\; \lambda' (X;H) \in T_{\Th} (\lambda(X))  \big\}.
\end{equation*}
\end{Corollary}
\begin{proof} Taking the symmetric set $\Th$ from \eqref{spectralset}, we can apply Theorem~\ref{specsubd} for the symmetric function $\dd_\Th$.
The claimed representation of the tangent cone to $C$ at $X$ follows from the facts that  $\d \dd_\Th(X)=\dd_{T_C(X)}$ and $ \d \dd_\Th (\lambda(X))=\dd_{T_{\Th} (\lambda(X))}$.
\end{proof}

\begin{Example}[tangent cone to $\S^n_-$]\label{tans}{\rm Suppose that $\S^n_-$ stands for the cone of all $n\times n$ symmetric and negative semidefinite matrices. This cone is a spectral set and  
 $$
\S^n_-=\big\{X\in \S^n\big|\; \lm(X)\in \R^n_-\big\}.
$$
Take $X\in \S^n_-$ and assume that  $ \mu_{1} >  \cdots>  \mu_{r} $ are its distinct eigenvalues. If $ \mu_1<0$, then
we clearly have $\lm(X)\in \inte \R^n_-$ and hence $T_{ \R^n_-}(\lm(X))=\R^n$. Using this, together with Corollary~\ref{tansp}, we get 
$T_{\S^n_-}(X)=\S^n$. If $  \mu_1=0$, then we obtain $T_{ \R^n_-}(\lm(X))=\R^{|\al_1|}_-\times \R^{n-|\al_1|}$, where $\al_1$
is defined by \eqref{index}. Appealing to Corollary~\ref{tansp} tells us that 
\begin{equation}\label{1ts}
T_{\S^n_-}(X)=\big\{H\in \S^n\big|\; \lm_1(U_{\al_1}^\top HU_{\al_1})\le 0\big\},
\end{equation}
where $U$ is taken from \eqref{specdocom}. 
}
\end{Example}
The tangent cone description for the set $\S^n_-$ in \eqref{1ts} can be alternatively obtained by \cite[Proposition~2.61]{bs}. Indeed, 
it is easy to see that $\S^n_- = \{ X \in \S^n \big|\; \lambda_{1}(X) \leq 0  \}$ in which $\lambda_1$  is known to be a  convex function (cf. \cite[Exercise~2.54]{rw}). Obviously, we can find $ {X} \in \S^n$ 
with  $\lambda_1 ( {X}) < 0$, a condition  known as the {\em Slater} condition and assumed in \cite[Proposition~2.61]{bs}. In contrast,  our approach relies upon  the   metric subregularity, automatically satisfied for  spectral sets. 
This allows to calculate the tangent cone of spectral sets even if the Slater condition  fails therein as  the following example demonstrates. 
\begin{Example}[failure of the Slater condition in spectral sets]\label{fs}{\rm Assume that    $k \in \N$ and consider the set
\begin{equation}\label{fs1}
C =\big\{ X \in \S^n_+ \big|\; \sum^{n}_{i=1} \lambda^{k}_{i} (X) = 1 \big\}.
\end{equation} 
This is clearly a  spectral subset of $\S^n$. When $k = 1$, the  set $C$  is called { \em spectahedron}.  
Note that $C$ can be represented in the form of  \eqref{spectralset} with the symmetric set $\Th$ defined by 
 \begin{equation}\label{sym}
 \Theta  =\big\{ (z_1,\ldots,z_n) \in \R^n \big|\; \sum^{n}_{i=1} z^{k}_{i}  = 1, \; z_i \geq 0  \;\;\mbox{for all}\;\; i\in\{1,\ldots,n\} \big\}.
 \end{equation} 
 Set $\Phi(z)=(\sum^{n}_{i=1} z^{k}_{i}-1,z_1,\ldots,z_n)$ with $z=(z_1,\ldots,z_n)$ and $D=\{0\}\times\R_+^n$ and observe that 
 \begin{equation}\label{cons}
 \Th=\{z\in \R^n|\; \Phi(z)\in D\}.
 \end{equation} 
 We claim now that 
 \begin{equation}\label{bcq}
 N_D(\Phi(z))\cap \ker\nabla\Phi(z)^*=\{0\}
 \end{equation}
  for any $z\in \Th$. To justify it, pick $z=(z_1,\ldots,z_n)\in \Th$ and assume that $(b_0,\ldots,b_n)\in N_D(\Phi(z))\cap \ker\nabla\Phi(z)^*$.
 This implies that $b_iz_i=0$ and $kz_i^{k-1}b_0+b_i=0$ for any $i=1,\ldots,n$. It is not hard to see that these conditions lead us to $b_i=0$ for any $i=0,\ldots,n$, which proves our claim. 
 Take $X \in C$ and define the active index set $I(\lm(X)) =\big\{i \in \{1,\ldots,n\} |\; \lambda_i (X)=0 \big\}$.  
 It follows from  \cite[Theorem~6.14]{rw} and \eqref{bcq} that the tangent cone to $\Th$ at $\lm(X)$ can be calculated as  
 $$
 T_{\Theta} (\lambda(X)) =\big\{ (w_1,\ldots,w_n) \in \R^n \big|\; \sum^{n}_{i=1} w_i \lambda^{k-1}_{i} (X) = 0, \; w_i \geq 0 \;\;\mbox{for all}\;\;i \in I(\lm(X)) \big\}.
 $$
Appealing to Corollary~\ref{tansp} tells us that 
\begin{equation}\label{fs2}
T_{C}(X)=\big\{H\in \S^n\big|\; \sum^{n}_{i=1} \lm'_{ i}(X; H) \lambda^{k-1}_{i} (X) = 0, \; \ \lm'_{ i}(X; H) \geq 0 \;\;\mbox{for all}\;\;i \in I(\lm(X)) \big\}.
\end{equation}
Note that due to the presence of the equality constraint, the Slater condition fails for  $C$ and  thus \cite[Proposition~2.61]{bs} can't be applied. 
}
\end{Example}

\section{ Parabolic Epi-Differentiability  of Spectral Functions}\sce \label{sect04}

The main objective of this section is to provide a systematic study of two important second-order variational properties of spectral sets and functions: 1) parabolic derivability; and 2) a chain rule for parabolic subderivatives.
To achieve these goals, we begin with justifying that certain second-order  approximations of  spectral sets enjoy an outer Lipschitzian property, which is central to our developments in this section.  
Suppose that $C\subset \S^n$ is a spectral set with the representation in \eqref{spectralset} and that $X\in C$ and $H\in T_C(X)$. Define the set-valued mapping $S_{H} : \R^n \tto \S^n$ via the second-order tangent set 
to the symmetric set  $\Th$ in \eqref{spectralset} by
\begin{equation}\label{mapt}
S_H (p):=\big\{W \in \S^n \big|\;\lambda^{''} (X;H,W) + p \in T^2_{\Theta}  \big(\lambda(X), \lambda' (X ; H) \big)\big\}.
\end{equation}
For any parameter $p\in \R^n$, the set-valued mapping $S_H(p)$  presents a  second-order tangential approximation of the spectral set \eqref{spectralset} at $X$ for $H$. 
Note that by Corollary~\ref{tansp}, the condition $H\in T_C(X)$ amounts to $\lambda' (X;H) \in T_{\Th} (\lambda(X))$, which is required in the 
definition of the second-order tangent set to $\Th$ at $\lm(X)$ in \eqref{mapt}; see \eqref{2tan}. Note also that reducing  the estimate \eqref{specmsqc}, 
which was stated for the spectral function in \eqref{spec}, to the spectral set $C$ in \eqref{spectralset} gives us the estimate 
\begin{equation} \label{mscq2}
\dist ( X, C) = \dist ( \lambda(X),\Th )\quad \mbox{for all}\;\; X\in \S^n,
\end{equation}
which will be utilized broadly in this section.

\begin{Proposition}[uniform outer Lipschitzian property of $S_H$]\label{fors}Assume that $C\subset \S^n$ is a spectral set with the representation \eqref{spectralset} and that $X\in C$ and $H\in T_C(X)$.
Then the   mapping $S_H$ in \eqref{mapt} enjoys  the following uniform outer Lipschitzian property  at the origin:
\begin{equation}\label{tlip}
S_H (p)\subset S_H (0)+ \|p\| \B\;\mbox{ for all }\;p\in\R^n.
\end{equation}
 \end{Proposition}
\begin{proof} Let  $p\in \R^n$ and   pick then $ W \in S_H(p)$. It follows from  \eqref{mapt} that
$
\lambda^{''} (X;H,W) + p \in T^2_{\Theta}  \big(\lambda(X), \lambda' (X ; H) \big)
$.
By  \eqref{2tan}, there exists a sequence $t_k\dn 0$ such that 
\begin{equation*}
\lambda (X)+t_k\lambda' (X ; H)  +\sm t_k^2\lambda^{''} (X;H,W) + \sm t_k^2 p +o(t^2_k)\in\Th\quad \mbox{for all}\;\;k\in\N.
\end{equation*}
For any $k$ sufficiently large, we conclude from    \eqref{secondexp2} that
\begin{equation*}
 \lambda(X+t_kH + \sm t_k^2 W) =  \lambda (X)+t_k\lambda' (X ; H)  +\sm t_k^2\lambda^{''} (X;H,W) +  o(t^2_k),
\end{equation*}
which in turn implies via \eqref{mscq2} that
\begin{eqnarray*}\label{pl01}
{\rm dist}\big(X+t_kH + \sm t_k^2 W,C \big)= {\rm dist} \big(\lambda(X+t_kH + \sm t_k^2 W),\Theta \big)\le \sm t_k^2\|p\|+o(t_k^2).
\end{eqnarray*}
This ensures the existence of a matrix $Y_k\in C$ such that 
\begin{equation*}
\|D_k\| \le  \frac{1}{2} \Big(\|p\|+\frac{o(t_k^2)}{t_k^2}\Big)\;\mbox{ with }\; D_k:=\frac{X + t_k  H +\sm t_k^2 W - Y_k}{t_k^2}.
\end{equation*}
Passing to a subsequence, if necessary, ensures the existence of $D \in \S^n$ such that $D_k\to D$ as $k\to\infty$. This yields the estimate
\begin{equation}\label{pl02}
\|D\| \le \sm  \|p\|.
\end{equation}
It follows from  $ X +t_k H +\sm  t_k^2 W -t_k^2 D_k= Y_k\in C$ and \eqref{spectralset}  that  $\lambda(X + t_k  H +\sm t_k^2 W -t_k^2 D_k)\in \Th$. Taking into account \eqref{secondexp2}, 
we get for any $k\in \N$ sufficiently large that 
\begin{equation*}
\lambda(X + t_k  H +\sm t_k^2 W -t_k^2 D_k)=\lambda(X)+t_k \lambda' (X ;H) +\sm t_k^2  \lambda'' \big( X ; H , W-2D \big)+o(t_k^2) \in \Th.
\end{equation*}
By the definition of the second-order tangent set, we arrive at 
\begin{equation*}
   \lambda'' \big( X ; H , W-2D \big) \in T_{\Th}^{2} (\lambda(X) , \lambda (X ; H)),
\end{equation*}
which yeilds  $W-2D\in S_H (0)$. This, combined with \eqref{pl02}, justifies the claimed inclusion in  \eqref{tlip} and thus completes the proof.
\end{proof} 

The outer Lipschitzian property for second-order tangential approximations  appeared first in \cite[Theorem~4.3]{mms2} for 
  sets $C$ as the one in \eqref{spectralset} with the eigenvalue function $\lm(\cdot)$ replaced with a twice differentiable function,  
  under an adaptation of  the constraint qualification \eqref{mscq} for this setting.  Proposition~\ref{fors} demonstrates  
  that the latter result can be achieved  without the assumed twice differentiability in \cite{mms2} when we still have a second-order expansion 
  for functions in our settings. 
  
  Next, we are going to achieve a chain rule for second-order tangent sets of spectral sets, which heavily relies upon  Proposition~\ref{fors}. First, we recall 
  \cite[Theorem~4.5]{mms2}, where a similar result was proven for constraint systems in finite dimensional  Hilbert spaces.
  \begin{Proposition}\label{2tcr}
  Let $D$ be a closed subset of   $\Y$ and  let  $\Omega =\big\{x\in \X\big|\; \Phi(x)\in D\big\}$, where $\Phi:\X\to \Y$ is a twice differentiable function between two Euclidean spaces, and $\ox\in \Omega$. Suppose further that there are $\kappa\ge 0$ and $\ve>0$  such that the estimate  
\begin{equation}\label{nch8}
 {\dist}(x, \Omega)\le\kappa\,{\rm dist}\big(\Phi(x),D\big)\;\mbox{ for all }\;x\in \B_\ve(\ox)
\end{equation}
holds.  Then for all $w\in  T_ \Omega(\ox)$, we have  
\begin{equation}\label{nch2}
  T^2_ \Omega(\ox,w)=\big\{u\in \X|\; \nabla \Phi(\ox)u+\nabla^2 \Phi(\bar x)(w,w)\in T^2_D\big(\Phi(\bar x),\nabla \Phi(\ox)w\big)\big\}.
\end{equation}
If furthermore the set $D$ is parabolically derivable at $\Phi(\ox)$ for $\nabla \Phi(\ox)w$, then the constraint set  $\Omega$  is parabolically derivable at $\ox$ for $w$.
  \end{Proposition}
  \begin{proof}
  The equality in \eqref{nch2} was justified in \cite[Theorem~4.5]{mms2}  under an extra assumption that  the set $D$ is regular in the sense of \cite[Definition~6.4]{rw}; 
  see \cite[Proposition~5]{gyz} for an extension of \cite[Theorem~4.5]{mms2} without the  regularity assumption on $D$. Note that the directional metric subregularity used 
  in\cite[Proposition~5]{gyz} is weaker than \eqref{nch8} in general. However, it was shown in \cite[Lemma~2.8(ii)]{go} that metric subregularity at any direction is 
  equivalent to \eqref{nch8}. 
  
 To prove parabolic derivability of $\Omega$ at $\ox$ for $w$, pick $u\in  T^2_ \Omega(\ox,w)$. It follows from the proof of  \cite[Proposition~5]{gyz}  that 
  there is a positive constant $\ell$ such that 
\begin{equation}\label{mmm}
  \dist\big( u, \frac{\Omega-\ox-tw}{\sm t^2}\Big)\le \ell\, \dist\Big(\nabla \Phi(\ox)u+\nabla^2 \Phi(\bar x)(w,w), \frac{D-\Phi(\ox)-t\nabla \Phi(\ox)w}{\sm t^2}\Big)+\frac{o(t^2)}{t^2}
\end{equation}
  for any $t$ sufficiently small that $t\dn 0$. Since $u\in  T^2_ \Omega(\ox,w)$, we conclude from \eqref{nch2} that $\nabla \Phi(\ox)u+\nabla^2 \Phi(\bar x)(w,w)\in  T^2_D\big(\Phi(\bar x),\nabla \Phi(\ox)w\big)$,
  which together with  parabolic derivability of $D$ at $\Phi(\ox)$ for $\nabla \Phi(\ox)w$ implies via  \cite[Corollary~4.7]{rw} that 
  $$
  \dist\Big(\nabla \Phi(\ox)u+\nabla^2 \Phi(\bar x)(w,w), \frac{D-\Phi(\ox)-t\nabla \Phi(\ox)w}{\sm t^2}\Big)\to 0 \quad \mbox{and}\quad t\dn 0.
  $$
  This, coupled with \eqref{mmm}, confirms that for any sufficiently small $t$, there exists $u(t)\in (\Omega-\ox-tw)/\sm t^2$ such that $u(t)\to u$ as $t\dn 0$. 
  Define the arc $\xi(t):=-\ox+tw+\sm t^2 u(t)$ and observe that $\xi(0)=\ox$, $\xi'_+(0)=w$, and $\xi''_+(0)=u$. To finish the proof, we need to show that 
  $T^2_ \Omega(\ox,w)\neq \emptyset$, which was already established  in \cite[Corollary~1]{gyz} under the metric subregualrity condition in \eqref{nch8}.
  Combining these confirms that  $\Omega$  is parabolically derivable at $\ox$ for $w$ and hence completes the proof.
  \end{proof}

\begin{Theorem}[second-order tangent sets of spectral sets]\label{fsch} Assume that $C\subset \S^n$ is a spectral set with the representation in \eqref{spectralset} and that $X\in C$ and $H\in T_C(X)$. Then we have
\begin{equation}\label{chrs}
 T^2_{C}(X ,H) =\big\{W\in \S^n\big|\;  \lambda'' \big( X ; H , W \big) \in T^2_{\Th}\big (\lambda(X) , \lm'(X;H) \big)\big\},
\end{equation}  
where $\Th$ is taken from \eqref{spectralset}.
Moreover,  the following properties are satisfied.
\begin{itemize}[noitemsep,topsep=2pt]
\item [ \rm {(a)}] If the symmetric set $\Th$   is parabolically derivable at $\lambda(X)$ for $\lambda' (X ; H)$, then $C$ is parabolically derivable at $X$ for $H$.
\item [ \rm {(b)}] If the symmetric set $C$  is parabolically derivable at $X$ for any $H\in T_C(X)$, then $\Th$ is parabolically derivable at $\lm(X)$ for any $v\in T_\Th(\lm(X))$.
\end{itemize}
\end{Theorem}
\begin{proof} 
First note from Corollary~\ref{tansp} that  the condition $H\in T_C(X)$ amounts to $\lambda' (X;H) \in T_{\Th} (\lambda(X))$.  Let $W \in \S^n$. Employing    \eqref{mscq2} and  \eqref{secondexp2}  tells us that for any 
$t>0$ sufficiently small, we have 
\begin{eqnarray}
\dist \big( X + t H + \sm t^2 W,C  \big) &=& \dist \big(\lambda ( X + t H + \sm t^2 W),\Th  \big) \nonumber\\
&=& \dist \big( \lambda (X)+ t \lambda' (X ; H)  +\sm t^2 \lambda^{''} (X;H,W) + o(t^2), \Th \big). \label{chrs1}
\end{eqnarray}
Take $W \in T^2_{C} (X , H)$. By \eqref{2tan},  there exists a sequence $t_k \dn 0$ such that $ X + t_k H + \frac{1}{2} t_k^2 W + o(t_k^2) \in  C$.
By \eqref{chrs1}, we get $ \lambda (X)+ t_k \lambda' (X ; H)  +\sm t_k^2 \lambda^{''} (X;H,W) + o(t_k^2) \in \Th$, which clearly demonstrates that 
$ \lambda'' \big( X ; H , W \big) \in T^2_{\Th}\big (\lambda(X) , \lm'(X;H) \big)$ and thus proves the inclusion ``$\subset$" in \eqref{chrs}.
The opposite inclusion in \eqref{chrs} can be established via   a similar argument and \eqref{chrs1}, which proves the claimed representation 
of the second-order tangent set to $C$ in \eqref{chrs}.

To prove (a),  suppose that the symmetric set $\Th$ is parabolically derivable at $\lambda(X)$ for $\lambda' (X ; H)$. To justify the same property for $C$ at $X$ for $H$, pick $W \in T^2_C (X ,H)$. 
By \eqref{chrs}, we obtain  $\lambda^{''} (X;H,W) \in T^2_{\Th} (\lambda(X) , \lambda' (X ; H) )$. Since $\Th$ is parabolically derivable at $\lambda(X)$ for $\lambda' (X ; H)$, we
conclude from  Proposition~\ref{pdch} that there exists $\ve>0$ such that for all $t\in [0,\ve]$, we have  
$$ 
\lambda (X)+ t \lambda' (X ; H)  +\sm t^2 \lambda^{''} (X;H,W) + o(t^2) \in \Th.
$$ 
Reducing $\ve>0$ if necessary,  pick $t\in [0,\ve]$ and conclude from \eqref{chrs1} that $X + t H + \frac{1}{2} t^2 W + o(t^2)  \in C$. 
Defining the arc $\xi:[0,\ve]\to C$ by $\xi(t)=X + t H + \frac{1}{2} t^2 W + o(t^2)$ for $t\in [0,\ve]$,
we can readily see that $\xi(0)=X$, $\xi_+'(0)=H$, and $\xi''_+(0)=W$.
To finish the proof of   parabilic derivability of $C$ at $X$ for $H$, it remains to show that $T^2_{C}(X,H)\neq \emptyset$. To this end, pick $Z  \in \S^n$ and $y \in T^2_{\Th} \big( \lambda(X) , \lambda' (X ; H) \big)$. In fact,
  such $y$ exists since $\Th$ is  parabolic derivable at $\lambda(X)$ for $\lambda' (X ; H)$. Therefore, we have 
\begin{equation*}
  \lambda^{''} (X;H,Z) + p \in T^2_{\Th} \big( \lambda(X) , \lambda' (X ; H) \big) \quad \mbox{ with }\;\; p:=y -  \lambda^{''} (X;H,Z),
\end{equation*}
which can be equivalently expressed as $Z\in S_H(p)$ via the mapping $S_H$ in \eqref{mapt}. Appealing to  Proposition~\ref{fors} and  the established outer Lipschitzian property in \eqref{tlip}, 
we find  a matrix  $W \in S_H(0)$ such that $\|Z-W\|\le \|p\|$. This tells us that
\begin{equation*}
  \lambda^{''} (X;H,W)  \in T^2_{\Th} \big( \lambda(X) , \lambda' (X ; H) \big).
\end{equation*}
Using the chain rule \eqref{chrs} leads us to $W \in T^2_{C}(X , H)$, and thus  $T^2_{C}(X , H)\neq \emptyset$. This shows that 
$C$ is parabolically derivable at $X$ for $H$, and thus  proves (a).

Turning into the proof of (b), observe first that $\Th$ can be represented as the constraint system  \eqref{setth}.  Adapting the estimate in \eqref{revset2} for the latter constraint system gives 
us the estimate  
$$
\dist(x,\Th)=\dist(\Diag(x),C)\quad \mbox{for all}\;\; x\in \R^n.
$$
This, together with twice differentiability of the mapping $x\mapsto \Diag(x)$ with $x\in \R^n$, allows us to conclude from 
\eqref{nch2} that for any $v\in T_\Th(\lm(X))$ we always have 
\begin{equation*}
w\in T^2_\Th (\lm(X),v)\iff \Diag(w)\in T^2_C\big(\Diag(\lm(X)),\Diag(v)\big).
\end{equation*}
 To justify (b), pick $v\in T_\Th(\lm(X))$. We are going to show that $\Th$ is 
parabolically derivable at $\lm(X)$ for $v$. According to  Proposition~\ref{2tcr}, this will be ensured provided that $C$ is parabolically derivable at $\Diag(\lm(X))$ for $\Diag(v)$.
Since $C$ is a spectral set, it is easy to see that 
\begin{equation}\label{vb01}
W\in T^2_C\big(\Diag(\lm(X)),\Diag(v)\big)\iff UWU^\top\in T_C^2\big(X, U\Diag(v)U^\top\big),
\end{equation}
where $U$ is taken from \eqref{specdocom}. Moreover, it follows from $v\in T_\Th(\lm(X))$ and Proposition~\ref{chss} that $\Diag(v)\in T_C\big(\Diag(\lm(X))\big)$, which 
tells us that $U\Diag(v)U^\top \in T_C(X)$. By assumption, we know that $C$  is parabolically derivable at $X$ for $U\Diag(v)U^\top$. 
This, combined with 
\eqref{vb01}, confirms that  $C$  is parabolically derivable at $\Diag(\lm(X))$ for $\Diag(v)$. To justify this claim, take $W\in T^2_C\big(\Diag(\lm(X)),\Diag(v)\big)$.
By \eqref{vb01}, parabolic  derivability of $C$ at  $X$ for $UWU^\top$, and Proposition~\ref{pdch}, 
we find $\ve>0$ such that for all $t\in [0,\ve]$, the inclusion 
$$ 
X+ t U\Diag(v)U^\top +\sm t^2 UWU^\top + o(t^2) \in C
$$ 
is satisfied. It follows from  $C$ being  a spectral set and  the latter inclusion  that 
$$
\Diag(\lm(X))+  t\Diag(v)+ \sm t^2 W+o(t^2)=U^\top XU+ t\Diag(v)+ \sm t^2 W+o(t^2) \in C\quad \mbox{for all}\;\; t\in [0,\ve].
$$
Since $W\in T^2_C\big(\Diag(\lm(X)),\Diag(v)\big)$ was taken arbitrarily,  we conclude from Proposition~\ref{pdch} that $C$ is parabolically derivable at $\Diag(\lm(X))$ for $\Diag(v)$.
 Employing now  Proposition~\ref{2tcr} proves that 
$\Th$ is parabolically derivable at $\lm(X)$ for $v$ and hence completes the proof.
\end{proof}

\begin{Example}[second-order tangent set to $\S^n_{-}$]{\rm In the framework of Example~\ref{tans}, we are going to calculate 
the second-order tangent set to $\S^n_-$ at $X\in \S^n_-$ for any $H\in T_{\S^n_-}(X)$. To this end, we deduce from 
Theorem~\ref{fsch} that 
$$
 T^2_{\S^n_-}(X ,H) =\big\{W\in \S^n\big|\;  \lambda'' \big( X ; H , W \big) \in T^2_{\R^n_-}\big (\lambda(X) , \lm'(X;H) \big)\big\}.
$$
It follows from \cite[Proposition~13.12]{rw} that $\R^n_-$ is parabolically  derivable at   $\lm(X)$ for $ \lm'(X;H)$, which together with Theorem~\ref{fsch} implies that $\S^n_-$
enjoys the same property at $X$ for $H$. Moreover, we deduce from \cite[Proposition~13.12]{rw} that 
\begin{equation}\label{2ts}
 T^2_{\R^n_-}\big (\lambda(X) , \lm'(X;H) \big)=  T_{T_{\R^n_-}  (\lambda(X))} \big(\lm'(X;H) \big).
\end{equation}
If $\mu_1<0$,    we have $\lm(X)\in \inte \R^n_-$. This implies that   $T_{\R^n_-}  (\lambda(X))=\R^n$, which together with \eqref{2ts}  yields  $T^2_{\R^n_-}\big (\lambda(X) , \lm'(X;H) \big)=\R^n$
and thus $T^2_{\S^n_-}(X ,H) =\S^n$. Now, assume   that $\mu_1=0$.
According to Example~\ref{tans}, we have $T_{ \R^n_-}(\lm(X))=\R^{|\al_1|}_-\times \R^{n-|\al_1|}$, where $\al_1$
is defined by \eqref{index}. To proceed, since $H\in T_{\S^n_-}(X)$,  we need by \eqref{1ts}  to consider two cases: 1) $\lm_1(U_{\al_1}^\top HU_{\al_1})<0$;
and 2) $\lm_1(U_{\al_1}^\top HU_{\al_1})=0$. If the former holds, we obtain 
$$
T_{T_{\R^n_-}  (\lambda(X))} \big(\lm'(X;H) \big)=T_{\R^{|\al_1|}_-\times \R^{n-|\al_1|}} \big(\lm'(X;H) \big)=\R^n,
$$
 which together with \eqref{2ts} brings us again  to $T^2_{\R^n_-}\big (\lambda(X) , \lm'(X;H) \big)=\R^n$
and thus $T^2_{\S^n_-}(X ,H) =\S^n$. If    the latter holds, denote by $\eta_{1}^{1} >    \cdots> \eta_{\rho_1}^{1}$ the distinct eigenvalues of $U_{\al_1}^\top H U_{\al_1}$
and take the index set $\beta_1^1$ from \eqref{index2}. Recall that $|\beta_1^1|\le | \al_1|$.  Using this, we obtain 
$$
T_{T_{\R^n_-}  (\lambda(X))} \big(\lm'(X;H) \big)=T_{\R^{|\al_1|}_-\times \R^{n-|\al_1|}} \big(\lm'(X;H) \big)=\R^{|\beta_1^1|}_-\times \R^{n-|\beta_1^1|}.
$$
This, combined with \eqref{paralam} and \eqref{2ts}, leads us to 
\begin{eqnarray*}
 T^2_{\S^n_-}(X ,H) &=&\big\{W\in \S^n\big|\;  \lambda'' \big( X ; H , W \big) \in \R^{|\beta_1^1|}_-\times \R^{n-|\beta_1^1|} \big\}\\
 &=& \big\{W\in \S^n\big|\;  \lm_1 \Big( {R_{11}}^{\top}  \big( U_{\al_1}^\top ( W - 2  H X^{\dagger} H ) U_{\al_1} \big) R_{11} \Big) \le 0\big\},
\end{eqnarray*}
where $R_{11}=(Q_1)_{\beta_1^1}$ is taken from Proposition~\ref{secondexp}. We should point out that the second-order tangent set to $\S^n_-$ was calculated by finding the parabolic 
second-order directional derivative of the maximum eigenvalue function in \cite[page~474]{bs}; see also \cite[page~583]{zzx} for a different derivation of this object.  
}
\end{Example}

In the next example, we obtain the second-order tangent set to the spectral set defined in   \eqref{fs1}.  
Note  again that while obtaining such result by \cite[Proposition~3.92]{bs} requires the Slater condition, our approach shows that no constraint qualification is needed for this propose.  

\begin{Example}\label{2fs}{\rm Let $C$ be the spectral set   in   \eqref{fs1} and $X \in C$. Given $H \in T_C(X)$, we aim to determine  $T^2_C (X , H)$ using Theorem~\ref{fsch}. 
We know from Example~\ref{fs} that $C$ has the spectral representation in \eqref{spectralset} with   the symmetric set $\Theta$ defined by \eqref{sym}. Moreover, we showed that 
$\Th$ can be equivalently described as the constraint set in \eqref{cons} with $\Phi(z)=  (\sum^n_{i=1} z_i^k-1 , z  )$ for all $z=(z_1,\ldots,z_n) \in \R^n$. We deduce from \cite[Proposition~13.13]{rw} that
$w\in T^2_{\Theta} (\lambda(X) , \lambda' (X;H))$ if and only if we have 
\begin{equation}\label{1313}
  \nabla \Phi(\lambda(X)) w + \nabla^2 \Phi(\lambda(X)) \big( \lambda' (X;H) , \lambda' (X;H) \big) \in T^2_{\{ 0\} \times \R^n_+ } \big(\Phi(\lambda(X)) , \nabla \Phi(\lambda(X)) (\lambda' (X;H))\big).
\end{equation}
Using the index set $I(\lm(X))$ taken from  Example~\ref{fs}, define the  index set
$$ 
I\big(\lm(X),\lm'(X;H)\big) := \big\{ i \in I(\lm(X)) \big| \; \lm'_i(X;H) = 0 \big\}
$$
and conclude then from  \cite[Proposition~13.12]{rw} that 
\begin{eqnarray*}
T^2_{\{ 0\} \times \R^n_+ } \big(\Phi(\lambda(X)) , \nabla \Phi(\lambda(X)) (\lambda' (X;H))\big)= T_{T_{\{ 0\} \times \R^n_+ } (\Phi(\lambda(X)))}\big(\nabla \Phi(\lambda(X)) (\lambda' (X;H))\big)\\
=\big\{ (w_0,\ldots,w_n)   \big| \:  w_0=0,\; w_i \geq 0 \quad \mbox{for all}\;\;i \in I\big(\lm(X),\lm'(X;H)\big)  \big\}.
\end{eqnarray*}
This, coupled with  \eqref{1313},   yields $(w_1,\ldots,w_n) \in T^2_{\Theta} (\lambda(X) , \lambda' (X;H)) $ if and only if $w_i \geq 0$ for all $i \in I\big(\lm(X),\lm'(X;H)\big)$ and 
$$ \sum^n_{i=1} \lambda_i (X)^{k-1} w_i +   (k-1) \sum^n_{i=1} \lambda_i (X)^{k-2} \lm'_i(X;H)^2  = 0  .$$
Appealing now to Theorem \ref{fsch},  we conclude that $C$ is parabolically derivable at $X$ for $H$ and  that $W \in T^2_C (X , H)$ if and only if $\lambda_i^{''} (X; H, W)\geq 0$ for all $i \in I\big(\lm(X),\lm'(X;H)\big)$ and 
$$ 
\sum^n_{i=1} \lambda_i (X)^{k-1}\lambda_i^{''} (X; H, W)+   (k-1) \sum^n_{i=1} \lambda_i (X)^{k-2}  \lm'_i(X;H)^2  = 0.
$$
Note that when  $k=1$ for which $C$ reduces to the spectahedron,  the above equation simplifies as   $\sum^n_{i=1} \lambda_i^{''} (X; H, W) = 0$.
}
\end{Example}

We proceed with characterizing parabolic epi-differentiability of spectral functions.  We begin with 
recalling  the concept of the parabolic subderivative, introduced by  Ben-Tal and Zowe in \cite{bz1}. 
Let    $f:\X \to \oR$ and let  $\ox\in \X$ with $f(\ox)$ finite    and $w\in \X$ with $\d f(\ox)(w)$ finite.  
The {\em parabolic subderivative} of $f$ at $\ox$ for $w$ with respect to $z$ is defined by 
\begin{equation*}\label{lk02}
\d^2 f(\bar x)(w\verl z):= \liminf_{\substack{
   t\dn 0 \\
  z'\to z
  }} \dfrac{f(\ox+tw+\frac{1}{2}t^2 z')-f(\ox)-t\d f(\ox)(w)}{\frac{1}{2}t^2}.
\end{equation*}
Recall from  \cite[Definition~13.59]{rw} that $f$ is called {\em parabolically epi-differentiable} at $\ox$ for $w$ if  
$$\dom \d^2 f(\ox)(w \verl \cdot)=\big\{z\in \X|\,  \d^2 f(\ox)(w \verl z)<\infty \big\}\neq \emptyset,$$
 and for every $z \in \X$ and every sequence $t_k\dn 0$  there exists a  sequences $z_k\to z$ such that
\begin{equation}\label{pepi}
\d^2 f(\bar x)(w\verl z) = \lim_{k\to \infty} \dfrac{f(\ox+t_kw+\frac{1}{2}t_k^2 z_k)-f(\ox)-t_k\d f(\ox)(w)}{\frac{1}{2}t_k^2}.
\end{equation}
We say that $f$ is parabolically epi-differentiable at $\ox$ if it satisfies this condition at $\ox$ for any $w\in \X$ where $ \d f(\ox)(w)$ is finite. Note that 
  the inclusion $\dom \d f(\ox) \subset T_{\ss \dom f} (\ox)$ always holds and      equality happens when, in addition,  $f$ is locally Lipschitz
continuous  around $\ox$ relative to its domain; see \cite[Proposition~2.2]{ms}. 
A list of important functions, appearing in different classes of constrained and composite optimization problems, that are parabolically epi-differentiable at 
any points of their domains can be found 
in \cite[Example~4.7]{ms}. 
By definition, it is not hard to see that the inclusion 
\begin{equation}\label{inps}
\dom \d^2 f(\ox)(w \verl \cdot)\subset T_{\ss \dom f}^2 (\ox , w)
\end{equation}
always holds for any $w\in T_{\ss \dom f}(\ox)$. The following result, taken from \cite[Propositions~2.1 and 4.1]{ms}, presents conditions
under which we can ensure equality in the latter inclusion. 
\begin{Proposition}[properties of parabolic subderivatives]\label{praepidom}
Let $f: \X \to \oR$ be finite at $\ox$,  locally Lipschitz continuous around $\ox$ relative to its domain, 
and  parabolic epi-differentiable at $\ox$ for $w \in T_{\ss \dom f}(\ox)$. Then the following properties hold.
\begin{itemize}[noitemsep,topsep=2pt]
\item [ \rm {(a)}] $\dom \d f(\ox)= T_{\ss \dom f}(\ox) $ and   $\dom \d^2 f(\bar x)(w\verl .) = T_{\ss \dom f}^2 (\ox , w)$.
\item [\rm {(b)}] $\dom f$ is parabolically derivable at $\ox$ for $w$.
\end{itemize}
\end{Proposition}
The next result presents sufficient conditions under which spectral functions are  parabolically epi-differentiable. Moreover, it achieves a useful 
 formula for parabolic subderivatives of this class of functions.  

\begin{Theorem}[parabolic subderivatives of spectral function]\label{sochpar}
Let $\th:\R^n\to \oR$ be a symmetric function, which is locally Lipschitz continuous  relative to its domain. Let $X\in \S^n$ with $(\th\circ \lm)(X)$   finite. 
Then the following properties hold.
\begin{itemize}[noitemsep,topsep=2pt]
\item [ \rm {(a)}] If   $H\in T_{\ss \dom (\th\circ \lm)}(X)$ and $\th$ is parabolically  epi-differentiable  at $\lambda (X)$ for  $ \lambda' (X ; H)$, then $\th\circ \lm$ is parabolically epi-differentiable at $ X $ for  $H$ 
and  its  parabolic subderivative at $ X $ for  $H$ and its domain can be calculated, respectively,  by  
\begin{equation}\label{sochpar1}
\d^2 (\th\circ \lm)( X )(H \verll W) = \d^2 \th (\lambda(X))\big( \lambda' (X ; H ) \verlm \lambda^{''} (X ; H , W) \big)
\end{equation}
and 
\begin{equation} \label{domofpara}
\dom  \d^2 (\th\circ \lm)(X)( H\verll \; . \; )  =  T_{\ss \dom (\th\circ \lm)}^2 (X,H).
\end{equation}
Moreover, if $\th$ is lsc and convex, the parabolic subderivative  $W\mapsto \d^2 (\th\circ \lm)( X )(H \verll W)$ is a convex function.
\item [\rm {(b)}] If $\theta\circ \lm$ is    parabolically epi-differentiable at $ X $, then $\th$ is parabolically  epi-differentiable  at $\lambda (X)$.
\end{itemize}
  
\end{Theorem}
\begin{proof} To justify (a), we proceed concurrently to show that  $\th\circ \lm$  is parabolically  epi-differentiable  at  $ X $ for  $H$ and that  \eqref{sochpar1} and \eqref{pepi}  hold  for  $\th\circ \lm$.
To this end, set $g:=\th\circ \lm$ and pick $W\in \S^n$ and proceed  with considering two cases. 
Assume first that  $W \notin T_{\ss \dom g}^2 (X ,  H)$. Employing the inclusion in \eqref{inps} for $g$, we get  $\d^2 g(X)(H\verll W)=\infty$.
On the other hand, by  \eqref{domcs} and Theorem~\ref{fsch}, we obtain 
\begin{equation}\label{dom2g}
 T^2_{\ss\dom g}(X ,H) =\big\{W\in \S^n\big|\;  \lambda'' \big( X ; H , W \big) \in T^2_{\ss\dom\th}\big (\lambda(X) , \lm'(X;H) \big)\big\}.
\end{equation}
This, combined with  $W \notin T_{\ss \dom g}^2 (X ,  H)$, yields  $\lambda^{''} (X ; H , W) \notin T^2_{\ss \dom \theta }\big (\lambda(X) , \lambda' (X;H) \big)$.
Observe from Corolllary~\ref{tansp} and \eqref{domcs} that $H\in T_{\ss \dom g}(X)$ amounts to  $ \lambda' (X ; H) \in T_{\ss \dom \theta}( \lambda(X))$.
By  Proposition~\ref{praepidom}(a), we arrive at 
\begin{equation}\label{domgp}
\dom  \d^2 \theta( \lambda(X))( \lambda (X ; H) \verlm \cdot ) = T_{\ss \dom \th}^2 (\lambda(X) ,  \lambda' (X ; H)).
\end{equation}
Combining these tells us that 
$
\d^2 \theta( \lambda(X)) \big( \lambda' (X ; H) \verlm \lambda^{''} (X ; H , W)\big)=\infty,
$
which in turn justifies \eqref{sochpar1} for every $W \notin T_{\ss \dom g}^2 (X ,  H)$. To verify \eqref{pepi}   for $g$ in this case, consider an arbitrary sequence  $t_k\dn 0$, set $W_k:=W$ for all $k\in \N$,
and observe that   
\begin{eqnarray*}
\infty =\d^2 g(X)( H\verll W) &\le & \liminf_{k\to \infty}  \dfrac{g(X +t_k H +\frac{1}{2}t_k^2 W_k)-g(X)-t_k \d g(X)(W)}{\frac{1}{2}t_k^2}.
\end{eqnarray*} 
This clearly justifies \eqref{pepi} for all $W \notin T_{\ss \dom g}^2 (X ,  H)$. 

Turning now to the case $W \in T_{\ss \dom g}^2 (X ,  H)$, we observe that since $\th$ is parabolically  epi-differentiable  at $\lambda(X)$ for  $ \lambda'(X ; H) $, Proposition~\ref{praepidom}(b) tells us that $\dom \th$ is parabolically derivable at $\lambda(X)$ for  $ \lambda' (X ; H) $. We conclude from Theorem~\ref{fsch}(a) that $\dom g$ is parabolically derivable at $X$ for $H$. In particular, we have 
 \begin{equation}\label{pdfps}
 T_{\ss \dom g}^2 (X,  H)\neq \emptyset.
 \end{equation}
Pick now  $W \in T_{\ss \dom g}^2 (X,  H)$ and  consider then an arbitrary sequence  $t_k\dn 0$. Thus, by the definition of parabolic derivability,   we find a sequence $W_k\to W$ as $k\to \infty$ such that 
\begin{equation}\label{seqz}
X_k:=X+t_k H+\sm t_k^2 W_k= X+t_k H+\sm t_k^2 W+o(t_k^2)\in \dom g \quad \mbox{for all}\;\;k\in \N.
\end{equation}
Moreover, since $\theta$ is parabolically epi-differentiable at $\lambda (X)$ for $\lambda' (X ;H)$,  we find a sequence 
$w_k\to w := \lambda^{''} (X;H,W) $ such that 
\begin{eqnarray*}\label{peg}
&&\d^2 \th (\lambda(X)) ( \lambda' (X ;H) \verlm w)\\
&=& \lim_{k\to \infty}  \dfrac{\theta (\lambda(X) +t_k \lambda' (X ;H)  +\frac{1}{2}t_k^2 w_k)-\th(\lambda(X))-t_k\d \th(\lambda(X))(\lambda' (X ;H))}{\frac{1}{2}t_k^2}.
\end{eqnarray*}
It follows  from  \eqref{dom2g} and $W \in T_{\ss \dom g}^2 (X,  H)$  that  $w \in T^2_{\ss \dom \th}\big (\lambda(X), \lambda' (X ;H) \big)$. Combining this with  \eqref{domgp} tells us that    $ \d^2 \th (\lambda(X)) ( \lambda' (X ;H) \verlm w) <\infty$.
This implies that $y_k:=\lambda(X) +t_k \lambda' (X ;H) +\frac{1}{2}t_k^2 w_k \in \dom \th$ for all $k$ sufficiently large. 
 Using this together with   \eqref{dchain},  \eqref{seqz}, and \eqref{secondexp2},  we obtain 
\begin{eqnarray}
\d^2 g(X)( H\verll W) &\le & \liminf_{k\to \infty}  \dfrac{g(X +t_k H +\frac{1}{2}t_k^2 W_k)-g(X)-t_k \d g(X)(H)}{\sm t_k^2}\nonumber\\
&\le &  \limsup_{k\to \infty}  \dfrac{g(X +t_k H +\frac{1}{2}t_k^2 W_k)-g(X)-t_k \d g(X)(H)}{\sm t_k^2}\nonumber\\
&= &  \limsup_{k\to \infty}  \dfrac{\theta (\lambda(X_k))-\th(\lambda(X))-t_k\d \th(\lambda(X))(\lambda' (X ;H))}{\frac{1}{2}t_k^2} \nonumber\\
& \le &  \limsup_{k\to \infty}  \dfrac{\theta (y_k) -\th(\lambda(X))-t_k\d \th(\lambda(X))(\lambda' (X ;H))}{\frac{1}{2}t_k^2} + \limsup_{k\to \infty}\dfrac{\theta (\lambda(X_k))- \theta(y_k)}{\frac{1}{2}t_k^2} \nonumber\\
&\le & \d^2 \th (\lambda(X)) ( \lambda' (X ;H) \verlm w) +\limsup_{k\to \infty}\ell \| \lambda{''} (X ; H , W) - w_k +\frac{o(t_k^2)}{t_k^2} \| \nonumber\\
&=& \d^2 \th (\lambda(X)) ( \lambda' (X ;H) \verlm w)\label{psgf},
\end{eqnarray}
where $\ell\ge 0$ is a   Lipschitz constant of $\th$ around $\lambda(X)$ relative to its domain.
On the other hand, for any sequence $t_k\dn 0$ and any sequence $W_k\to W$, we can always conclude from \eqref{secondexp2} and  \eqref{dchain} that 
\begin{eqnarray*}
&& \liminf_{k\to \infty} 
  \frac{g(X +t_k H +\frac{1}{2}t_k^2 W_k)-g(X)-t_k \d g(X)(H)}{\sm t_k^2}\\
& =&  \liminf_{k\to \infty}  \dfrac{\theta \big(\lambda(X) +t_k \lambda' (X ;H)  +\frac{1}{2}t_k^2 \lambda^{''} (X;H,W) +o(t_k^2)\big)-\th(\lambda(X))-t_k \d \th(\lambda(X))(\lambda' (X ;H))}{\frac{1}{2}t_k^2}\\
&\ge&  \liminf_{\substack{
   t\dn 0 \\
  w'\to w
  }}   \dfrac{\theta \big(\lambda(X) +t \lambda' (X ;H)  +\frac{1}{2}t^2 w'  \big)-\th(\lambda(X))-t \d \th(\lambda(X))(\lambda' (X ;H))}{\frac{1}{2}t^2}\\
  &=& \d^2 \th (\lambda(X)) ( \lambda' (X ;H) \verlm w).
\end{eqnarray*}
This clearly yields the inequality 
$$
 \d^2 \th (\lambda(X)) ( \lambda' (X ;H) \verlm w) \le \d^2 g(X)(H\verll W)\quad \mbox{with}\;\; w=\lambda^{''} (X;H,W).
$$
Combining this and \eqref{psgf} implies that 
\begin{equation}\label{ps2}
 \d^2 g(X)(H\verll W)= \d^2 \th (\lambda(X)) ( \lambda' (X ;H) \verlm w)
\end{equation}
and that 
$$
\d^2 g(X)( H\verll W) = \lim_{k\to \infty}  \dfrac{g(X +t_k H +\frac{1}{2}t_k^2 W_k)-g(X)-t_k \d g(X)(H)}{\frac{1}{2}t_k^2},
$$
which  in turn prove both \eqref{sochpar1}  and \eqref{pepi} for any $W \in T_{\ss \dom g}^2 (X,H)$. As argued above, we also have 
$ \d^2 \th (\lambda(X)) ( \lambda' (X ;H) \verlm w) <\infty$, which together with \eqref{ps2} tells us that $  \d^2 g(X)( H\verll W)  < \infty $.
This brings us to the inclusion 
$$
T_{\ss \dom g}^2 (X,H)\subset \dom  \d^2 g(X)( H\verll \cdot).
$$
Since the opposite inclusion always holds (see \eqref{inps}), we arrive at \eqref{domofpara}. Combining this and   \eqref{pdfps} indicates that $\dom  \d^2 g(X)( H\verll \cdot)\neq \emptyset$,
and hence shows that $g$ is parabolically epi-differentiable at $ X $ for  $H$. Finally, assume that $\th$ is lsc and convex. By \cite[Corollary~5.2.3]{bl}, the spectral function $g=\th\circ \lm$
is convex. According to \cite[Example~13.62]{rw}, parabolic epi-differentiability of $g$ at $ X $ for  $H$ amounts to parabolic derivability of $\epi g$ at $ (X,g(X)) $ for  $(H, \d g(X)(H))$ and 
$$
\epi  \d^2 g(X)( H\verll \; . \; )=T^2_{\ss\epi g}\big( (X,g(X)), (H, \d g(X)(H))\big).
$$
Since $g$ is convex, it follows from parabolic derivability of $\epi g$ at $ (X,g(X)) $ for  $(H, \d g(X)(H))$ that $T^2_{\ss\epi g}\big( (X,g(X)), (H, \d g(X)(H))\big)$ is a convex set.
The above equality then confirms that  $W\mapsto \d^2 (\th\circ \lm)( X )(H \verll W)$ is a convex function and hence  completes the proof of (a).

Turning into the proof of (b), we conclude 
from  \eqref{spec} that the symmetric function $\th$ satisfies \eqref{spec2}, which means that 
$\th$ can be represented as a composite function of $g$ and the linear mapping $x\mapsto \Diag(x)$ with $x\in \R^n$. We also deduce from the imposed assumption on  $\th$ and the inequality in \eqref{lipei}
that $g$ is locally Lipschitz continuous relative to its domain. Pick $v\in \dom \d\th(\lm(X))=T_{\ss\dom \th}(\lm(X))$ and apply the chain rule in \eqref{dchain4} 
to the representation   \eqref{domcs} of $\dom \th$ to obtain $\Diag(v)\in T_{\ss\dom g}(\Diag(\lm(X)))$. To justify the parabolic epi-differentiability of $\th$ at $\lm(X)$ for $v$, we are going to use \cite[Theorem~4.4(iii)]{ms} by 
  showing that $g$ is parabolically epi-differentiable at $\Diag(\lm(X)) $ for $\Diag(v)$. To this end, it is not hard to see for any $U\in \O^n(X)$ that 
  \begin{equation}\label{psd4}
  \d g\big(\Diag(\lm(X))\big)\big(\Diag(v)\big)=\d g(X)(U\Diag(v)U^\top)
  \end{equation}
  Indeed, since $g$ is orthogonally invariant, we get   for any $U\in \O^n(X)$ that 
  \begin{eqnarray*}
\d g(X)(U\Diag(v)U^\top) &=& \liminf_{\substack{t\dn 0\\ W\to U\Diag(v)U^\top}}\frac{g \big( X+tW\big)-\th \big(X)\big)}{t}\\
&=& \liminf_{\substack{t\dn 0\\ U^\top WU\to \Diag(v)}}\frac{g \big(\Diag(\lambda (X))+tU^\top WU\big)-\th \big(\Diag(\lambda (X))\big)}{t}\\
&\ge &\d g\big(\Diag(\lm(X))\big)\big(\Diag(v)\big).
\end{eqnarray*}
A similar   argument leads us to  $\d g\big(\Diag(\lm(X))\big)\big(\Diag(v)\big)\le \d g(X)(U\Diag(v)U^\top) $ and thus proves \eqref{psd4}.
Similarly, we can show that 
 \begin{equation}\label{psd5}
  \d^2 g\big(\Diag(\lm(X))\big)\big(\Diag(v) \verll W\big)=\d^2 g(X)(U\Diag(v)U^\top \verll UWU^\top), \;\; W\in \S^n.
  \end{equation}
  Since  $g$ is a spectral function, $\dom g$ is a spectral set. This and  $\Diag(v)\in T_{\ss\dom g}(\Diag(\lm(X)))$ 
  tell us that $U\Diag(v)U^\top\in T_{\ss\dom g}(X)$. Since $g$ is parabolically epi-differentiable at $X$ for $U\Diag(v)U^\top$, the equality in \eqref{psd5} confirms that $g$ 
  enjoys the same property at $\Diag(\lm(X))$ for $\Diag(v)$. 
 Combining this,  \eqref{revset2}, and \cite[Theorem~4.4(iii)]{ms} shows that $\th$ is parabolically epi-differentiable at $ \lm(X) $ for $v$ and hence completes the proof.
\end{proof}

We close this section by revealing that the parabolic subderivative of spectral functions is symmetric with respect to a subset of $\P^n$, the set of all $n\times n$ permutation matrices.
This plays a central role in next section when we are going to study parabolic regularity of spectral functions. To this end, recall from Remark~\ref{sysub} that 
if $\th:\R^n\to \oR$ is a symmetric function and $X\in \S^n$ with $\th(\lm(X))$ finite, the subderivative function $\d \th(\lm(X))$ is a symmetric function with respect to 
$\P^n_X$, which is a subset of $\P^n$ consisting of all $n\times n$ block diagonal matrices in the form $Q=\Diag(P_1,\ldots,P_r)$, where $P_m\in \R^{|\al_m|\times |\al_m|}$, $m=1,\ldots,r$, is a  permutation matrix with $\al_m$ taken from \eqref{index} and $r$
being the number of distinct eigenvalues of $X$. Consider now $H\in \S^n$ with $\d\th(\lm(X))(\lm'(X;H))$ finite. Take the orthogonal matrix $U$ from \eqref{specdocom} and $m\in \{1,\ldots,r\}$.
Suppose that $\rho_m$ is the number of distinct eigenvalues of $U^\top_{\al_m}HU_{\al_m}$ and  pick then the index sets $\beta_j^m$ for $j=1,\ldots,\rho_m$ from \eqref{index2}.
Denote by $\P^n_{X,H}$ a subset of $\P^n_X$ consisting of all $n\times n$ matrices with representation   $\Diag(P_1,\ldots,P_r)$ such that  for each $m=1,\ldots,r$, the $|\al_m|\times |\al_m|$ permutation matrix $P_m$ has a block diagonal representation $P_m=\Diag(B^m_1,\ldots,B^m_{\rho_m})$,
where $B_j^m\in \R^{|\beta_j^m|\times |\beta_j^m|}$ is a permutation matrix for any $j=1,\ldots,\rho_m$. It is not hard to see that 
\begin{equation}\label{psym}
Q\lm(X)=\lm(X)\quad \mbox{and}\quad Q\lm'(X;H)=\lm'(X;H)\quad \mbox{for any}\;\; Q\in \P^n_{X,H}.
\end{equation}
\begin{Proposition}\label{psymm}
Assume that $\th:\R^n\to \oR$ is a symmetric function and $X\in \S^n$ with $\th(\lm(X))$ finite and that $H\in \S^n$ with $\d\th(\lm(X))(\lm'(X;H))$ finite. 
Then for any $w\in \R^n $ and any permutation matrix $Q\in \P^n_{X,H}$, we have 
$$
 \d^2 \th (\lambda(X))\big( \lambda' (X ; H ) \verlm Qw\big)= \d^2 \th (\lambda(X))\big( \lambda' (X ; H ) \verlm w\big),
$$
which means that the parabolic subderivative $w\mapsto \d^2 \th (\lambda(X))\big( \lambda' (X ; H ) \verlm w)$ is symmetric with respect to $ \P^n_{X,H}$.
\end{Proposition}
\begin{proof} Pick $w\in \R^n $ and  $Q\in \P^n_{X,H}$. Since $\th$ is symmetric, it follows from \eqref{psym} that 
\begin{eqnarray*}
&& \d^2 \th (\lambda(X))\big( \lambda' (X ; H ) \verlm w) \\
 &=& \liminf_{\substack{t\dn 0\\w'\to w}}\dfrac{\theta (\lambda(X) +t \lambda' (X ;H)  +\frac{1}{2}t^2 w')-\th(\lambda(X))-t\d \th(\lambda(X))(\lambda' (X ;H))}{\frac{1}{2}t^2}\\
&=& \liminf_{\substack{t\dn 0\\w'\to w}}\dfrac{\theta (\lambda(X) +t \lambda' (X ;H)  +\frac{1}{2}t^2 Qw')-\th(\lambda(X))-t\d \th(\lambda(X))(\lambda' (X ;H))}{\frac{1}{2}t^2}\\
&\ge & \liminf_{\substack{t\dn 0\\ v\to Qw}}\dfrac{\theta (\lambda(X) +t \lambda' (X ;H)  +\frac{1}{2}t^2 v)-\th(\lambda(X))-t\d \th(\lambda(X))(\lambda' (X ;H))}{\frac{1}{2}t^2}\\
&=& \d^2 \th (\lambda(X))\big( \lambda' (X ; H ) \verlm Qw).
\end{eqnarray*}
Similarly,  using \eqref{psym} for the matrix  $Q^{-1}=\Diag(P_1^{-1},\ldots,P_r^{-1})\in \P^n_{X,H}$,
one can conclude that  $ \d^2 \th (\lambda(X))\big( \lambda' (X ; H ) \verlm w)\le \d^2 \th (\lambda(X))\big( \lambda' (X ; H ) \verlm Qw)$, which justifies the claimed equality and hence ends the proof. 
\end{proof}

\section{ Parabolic Regularity of Spectral Functions}\label{sect05}  \sce
This section is devoted  to the study of  parabolic regularity of spectral functions whose central role in second-order variational analysis was revealed recently in \cite{mms2}.
As demonstrated in  \cite{mms2}, parabolic regularity can be viewed as an important {\em second-order regularity} with remarkable consequences among which we should highlight twice 
epi-differentiability of extended-real-valued functions. We begin with recalling the concepts of the second subderivative and parabolic regualrity for functions, respectively. 
Given a function $f: \X \to \oR$ and $\ox \in \X$ with $f(\ox)$ finite, define the parametric  family of 
second-order difference quotients for $f$ at $\ox$ for $\ov\in \X$ by 
\begin{equation*}\label{lk01}
\Delta_t^2 f(\bar x , \ov)(w)=\dfrac{f(\ox+tw)-f(\ox)-t\langle \ov,\,w\rangle}{\frac{1}{2}t^2}\quad\quad\mbox{with}\;\;w\in \X, \;\;t>0.
\end{equation*}
The {second subderivative} of $f$ at $\ox$ for $\ov$   is defined by 
\begin{equation*}\label{ssd}
\d^2 f(\bar x , \ov)(w)= \liminf_{\substack{
   t\dn 0 \\
  w'\to w
  }} \Delta_t^2 f(\ox , \ov)(w'),\;\; w\in \X.
\end{equation*}
The importance of the second subderivative resides in the fact that it can characterize the quadratic growth condition for optimization problems; see \cite[Theorem~13.24]{rw}. So, it is crucial
for many applications to calculate it in terms of the initial data of an optimization problem. This task was carried out for 
  major classes of functions including the convex piecewise linear-quadratic functions in the sense of \cite[Definition~10.20]{rw} in  \cite[Proposition~13.9]{rw}, 
the second-order/ice-cream cone in \cite[Example~5.8]{mms2}, the cone of positive semidefinite symmetric matrices in  \cite[Example~3.7]{ms}, and  the augmented Lagrangian of constrained optimization problems in  \cite[Theorem~8.3]{mms2}. 
We are going to calculate it for the spectral function $g$ in \eqref{spec} when the symmetric function $\th$ therein is convex.

\begin{Definition}[parabolic regularity]\label{pre} A convex function  $f:\X \to \oR$ is parabolically regular at $\ox$ for  $\ov\in \sub f(\ox)$ if 
 for any $w$ such that  $\d^2 f(\bar x , \ov)(w)<\infty  $, there exist, among the sequences $t_k\dn 0$ and $w_k\to w$ with 
$\Delta_{t_k}^2 f(\bar x , \ov)(w_k) \to \d^2 f(\bar x , \ov)(w)$, those with the additional property that 
$
\limsup_{k\to \infty}  {\|w_k-w\|}/{t_k}<\infty.
$
We say that  $f$ is parabolically regular at $\ox$ if it is  parabolically regular at $\ox$ for every $\ov\in \sub f(\ox)$.
A nonempty convex set $C\subset \X$ is said to be parabolically regular at $\ox$  if the indicator function $\dd_C$ is parabolically regular at $\ox$.
\end{Definition}

Parabolic regularity was introduced in \cite[Definition~13.65]{rw} for extended-real-valued functions but was not  scrutinized  therein.  
It was shown in \cite{mms2,ms} that polyhedral convex sets, the second-order/ice-cream cone, the cone of positive semidefinite symmetric matrices are parabolically regular.
One can also find in \cite[Corollary~13.68]{rw} that convex piecewise linear-quadratic functions are   parabolically regular. 
Recall that the critical cone of a convex function $f: \X \to \oR$  at  $\ox$ for $\ov\in \sub f(\ox)$ is defined  by 
\begin{equation*}\label{krit1}
K_f(\ox,\ov)= \{ w \in \X \:|\: \d f(\ox)(w) = \la \ov , w \ra \}.
\end{equation*}
When $f=\dd_C$, where $C$ is a nonempty convex subset of $\X$, the critical cone of $\dd_C$ at $\ox$ for $\ov$ is denoted by $K_C(\ox,\ov)$. In this case, the above definition of the critical cone of a function 
boils down to  the well known concept of the critical cone of a set (see \cite[page~109]{dr}), namely $K_C(\ox,\ov)=T_C(\ox)\cap [\ov]^\perp$  since $\d \dd_C(\ox)=\dd_{T_C(\ox)}$.

The following result is a special case of a more general characterization of parabolic regularity from \cite[Proposition~3.6]{ms} and will be utilized in our approach in this section. 
\begin{Proposition}[characterization of parabolic regularity] \label{parregularity} 
Assume that  $f:\X \to \oR$ is convex, finite at $\ox\in \X$, and  $\ov \in \sub f(\ox)$. Then  the following properties are equivalent.
\begin{itemize}[noitemsep,topsep=2pt]
\item [ \rm {(a)}] $f$ is parabolically regular at $\ox$ for  $\ov $
\item [ \rm {(b)}] For any $w\in K_f (\ox,\ov)$.  
\begin{equation*}\label{pri1}
\d^2 f (\bar x , \ov ) (w) =  \inf_{z\in \X}\big\{\d^2 f(\ox)(w\verl z) -\la z, \ov\ra\big \} 
\end{equation*}
\item [ \rm {(c)}]   For any $w\in \dom \d^2 f (\bar x , \ov )$, there exists a $\oz\in \dom \d^2 f(\ox)(w\verl \cdot)$ such that 
\begin{equation*}\label{pri11}
\d^2 f (\bar x , \ov ) (w) = \d^2 f(\ox)(w\verl \oz) -\la \oz, \ov\ra.
\end{equation*}
\end{itemize}
\end{Proposition} 
\begin{proof} The equivalence of (a) and (b)  and the implication (a)$\implies$(c) are taken  from  \cite[Proposition~3.6]{ms}. To prove (c)$\implies$(b), 
take $w\in K_f (\ox,\ov)$. It follows from \cite[Proposition~13.64]{rw} that the inequality `$\le$' in \eqref{pri1} is always satisfied. To prove the opposite inequality,
deduce  first   from the convexity of  $f$ that $\d^2 f (\bar x , \ov ) $ is proper due to $\d^2 f (\bar x , \ov )(0)=0$. By \cite[Proposition~13.5]{rw}, the inclusion 
$\dom \d^2 f (\bar x , \ov )\subset K_f (\ox,\ov)$  is satisfied. If $\d^2 f (\bar x , \ov )(w)=\infty$, the   inequality `$\ge$' in \eqref{pri1} trivially holds. Otherwise, 
$w\in  \dom \d^2 f (\bar x , \ov )$, which together with (c) proves  the   inequality `$\ge$' in \eqref{pri1} and hence completes the proof.
\end{proof}

Our results so far required   that the symmetric function $\th$ in \eqref{spec}  be locally Lipschitz continuous with respect to its domain. In what follows, we 
need to assume further that $\th$ is an lsc convex function. This assumption allows us to use the characterization of the subgradients  of the spectral function $g$ in \eqref{spec},
  recorded   in Proposition~\ref{subsp}. 
We begin   our analysis of the second subderivative of spectral functions by finding a lower estimate for it. 

\begin{Proposition}[lower estimate for second subderivatives] \label{lbs1} Assume that $g:\S^n\to \oR$ has the spectral representation in \eqref{spec} and is lsc and convex and that   $Y\in \sub g(X)$. 
Let $\mu_1>\cdots>\mu_r$ be the distinct eigenvalues of $X$  and $U\in \O^n(X)\,\cap\, \O^n(Y)$.  
Then for any $H\in \S^n$ we have 
\begin{equation}\label{lbss}
\d^2g(X,Y)(H)\ge \d^2\th\big(\lm(X),\lm(Y)\big)\big(\lm'(X;H)\big) + 2\sum_{m=1}^r \big\la \Lm(Y)_{\al_m\al_m},U_{\al_m}^\top  H ( \mu_{m} I  - X)^{\dagger} H U_{\al_m}   \big\ra,
\end{equation}
where  $\al_m$, $m=1,\ldots,r$, are defined in \eqref{index}.
\end{Proposition}
\begin{proof} Let  $H\in \S^n$ and   pick sequences $H_k\to H$ and $t_k\dn 0$. Setting $\Delta_{t_k}\lm(X)(H_k):=\big(\lm(X+t_kH_k)-\lm(X)\big)/t_k$, we get 
\begin{eqnarray*}
\Delta^2_{t_k} g\big(X,Y\big)(H_k)&=& \frac{\th\big(\lm(X+t_kH_k)\big)- \th\big(\lm(X)\big)-t_k\big\la Y,H_k\big\ra}{\sm t_k^2}\\
&=&  \frac{\th\big(\lm(X)+t_k\Delta_{t_k}\lm(X)(H_k)\big)- \th\big(\lm(X)\big)-t_k\big\la \lm(Y),\Delta_{t_k}\lm(X)(H_k)\big\ra}{\sm t_k^2}\\
&&+ \frac{\big\la \lm(Y),\Delta_{t_k}\lm(X)(H_k)\big\ra -\big\la Y,H_k\big\ra}{\sm t_k}\\
&=& \Delta^2_{t_k} \th\big(\lm(X),\lm(Y)\big)(\Delta_{t_k}\lm(X)(H_k)\big)+ \frac{\big\la \lm(Y),\Delta_{t_k}\lm(X)(H_k)\big\ra -\big\la Y,H_k\big\ra}{\sm t_k}.
\end{eqnarray*}
It follows from   $Y=U\Lm(Y)U^\top$ that
\begin{equation}\label{sdcy}
\big\la Y,H_k\big\ra= \big\la U\Lm(Y)U^\top,H_k\big\ra=\big\la \Lm(Y),U^\top H_kU\big\ra=  \sum_{m=1}^r \big\la \Lm(Y)_{\al_m\al_m},U_{\al_m}^\top  H_kU_{\al_m}\big\ra.
\end{equation}
On the other hand, it results from  \eqref{secondexp1} and     Fan's inequality that 
\begin{eqnarray*}
\big\la \lm(Y),\Delta_{t_k}\lm(X)(H_k)\big\ra &=& \sum_{m=1}^r\sum_{j\in \al_m} \frac{\lm_j(Y) \big( \lm_j(X+t_kH_k)-\lm_j(X)\big)}{t_k}\\
&=&\sum_{m=1}^r\sum_{j\in \al_m} \lm_j(Y) \lm_{\ell_j}\big(U_{\al_m}^\top H_kU_{\al_m}+t_kU_{\al_m}^\top  H_k( \mu_m I-X)^\dagger H_kU_{\al_m}\big) + O(t_k^2)\\
&\ge & \sum_{m=1}^r\big\la \Lm(Y)_{\al_m\al_m}, U_{\al_m}^\top H_kU_{\al_m}+t_kU_{\al_m}^\top  H_k( \mu_m I-X)^\dagger H_kU_{\al_m}\big\ra + O(t_k^2).\\
\end{eqnarray*}
Combining this with \eqref{sdcy} brings us to 
\begin{eqnarray*}
\frac{\big\la \lm(Y),\Delta_{t_k}\lm(X)(H_k)\big\ra -\big\la Y,H_k\big\ra}{\sm t_k}&\ge &  2\sum_{m=1}^r\big\la \Lm(Y)_{\al_m\al_m},  U_{\al_m}^\top  H_k( \mu_m I-X)^\dagger H_kU_{\al_m}\big\ra + O(t_k).
\end{eqnarray*}
This leads us to the estimate 
\begin{eqnarray*}
\Delta^2_{t_k} g\big(X,Y\big)(H_k)&\ge & \Delta^2_{t_k} \th\big(\lm(X),\lm(Y)\big)(\Delta_{t_k}\lm(X)(H_k)\big) \\
&& + 2\sum_{m=1}^r\big\la \Lm(Y)_{\al_m\al_m},  U_{\al_m}^\top  H_k( \mu_m I-X)^\dagger H_kU_{\al_m}\big\ra + O(t_k),
\end{eqnarray*}
which in turn clearly justifies the lower estimate in \eqref{lbss} for the second subderivative of $g$ at $X$ for $Y$
since $\Delta_{t_k}\lm(X)(H_k)\to \lm'(X;H)$ as $k\to \infty$. 
\end{proof}

We proceed with a result about the critical cone of spectral functions. 

\begin{Proposition}[critical cone of spectral functions] \label{crit} Assume that $g:\S^n\to \oR$ has the spectral representation in \eqref{spec} and is   lsc and convex and that   $Y\in \sub g(X)$. 
Let $\mu_1>\cdots>\mu_r$ be the distinct eigenvalues of $X$. Then, we have   $H\in K_g(X,Y)$ if and only if 
  $\lm'(X;H)\in K_\theta\big(\lm(X),\lm(Y)\big)$ and the matrices $\Lm(Y)_{\al_m\al_m}$ and $U_{\al_m}^\top  HU_{\al_m}$ have a simultaneous ordered spectral decomposition for any $m=1,\ldots,r$ with $\al_m$ taken from \eqref{index}
  and $U\in \O^n(X)\,\cap\, \O^n(Y)$. 
\end{Proposition}
\begin{proof} By \cite[Corollary~5.2.3]{bl}, the symmetric function  $\th$ in \eqref{spec} is lsc and convex. Thus, we find $U\in \O^n(X)\,\cap\, \O^n(Y)$ and get $\lm(Y)\in \sub \th \big(\lm(X)\big)$ due to Proposition~\ref{subsp}.
Pick  $H\in K_g(X,Y)$ and deduce from the definition of the critical cone that $\d g(X)(H)=\big\la Y,H\big\ra$. We then conclude  from $Y=U\Lm(Y)U^\top$ with $\Lm(Y)=\Diag(\lm(Y)\big)$ and Fan's inequality that 
\begin{eqnarray}
 \big\la Y,H\big\ra &=&  \big\la \Lm(Y),U^\top HU\big\ra=  \sum_{m=1}^r \big\la \Lm(Y)_{\al_m\al_m},U_{\al_m}^\top  HU_{\al_m}\big\ra\nonumber\\
&\le &  \sum_{m=1}^r \big\la \lm(Y)_{\al_m}, \lm(U^\top_{\al_m} H U_{\al_m})\big\ra =\big\la \lm(Y),\lm'(X;H)\big\ra\nonumber\\
&\le & \d \th\big(\lm(X)\big)\big(\lm'(X;H)\big)=\d g(X)(H)\label{cri4},
\end{eqnarray}
where the last inequality results from $\lm(Y)\in \sub \th \big(\lm(X)\big)$, \cite[Exercise~8.4]{rw}, and the convexity of $\th$ and where the last equality comes from \eqref{dchain}.
These relationships clearly imply that $\d \th\big(\lm(X)\big)\big(\lm'(X;H)\big)=\big\la \lm(Y),\lm'(X;H)\big\ra$, meaning that $\lm'(X;H)\in K_\theta (\lm(X),\lm(Y) )$, and that 
$\big\la \Lm(Y)_{\al_m\al_m},U_{\al_m}^\top  HU_{\al_m}\big\ra=\big\la \lm(Y)_{\al_m}, \lm(U^\top_{\al_m} H U_{\al_m})\big\ra $  for any $m=1,\ldots,r$. By Fan's inequality, these equalties are  equivalent to 
saying that the matrices $\Lm(Y)_{\al_m\al_m}$ and $U_{\al_m}^\top  HU_{\al_m}$ have a simultaneous ordered spectral decomposition for any $m=1,\ldots,r$. 

To prove the opposite claim, assume $\lm'(X;H)\in K_\theta\big(\lm(X),\lm(Y)\big)$ and the matrices $\Lm(Y)_{\al_m\al_m}$ and $U_{\al_m}^\top  HU_{\al_m}$ have a simultaneous ordered spectral decomposition for any $m=1,\ldots,r$. 
The latter tells us via Fan's inequality that $\big\la \Lm(Y)_{\al_m\al_m},U_{\al_m}^\top  HU_{\al_m}\big\ra=\big\la \lm(Y)_{\al_m}, \lm(U^\top_{\al_m} H U_{\al_m})\big\ra $  for any $m=1,\ldots,r$.
Moreover, the former yields $\d \th\big(\lm(X)\big)\big(\lm'(X;H)\big)=\big\la \lm(Y),\lm'(X;H)\big\ra$. Taking these into account demonstrates that both inequalities in \eqref{cri4} are indeed equalities.
This leads us to $\d g(X)(H)=\big\la Y,H\big\ra $, which implies that $H\in K_g(X,Y)$.
\end{proof}

The characterization of the critical cone of spectral functions, obtained above, is a generalization of a similar result, established recently in \cite[Proposition~4]{cd} for 
the spectral function $g$ from \eqref{spec} when the symmetric function $\th$ therein is a polyhedral function, meaning a function that  its epigraph is a polyhedral convex set.
To obtain a full characterization of the critical cone of spectral functions, we should know when the matrices $\Lm(Y)_{\al_m\al_m}$ and $U_{\al_m}^\top  HU_{\al_m}$ in 
Proposition~\ref{crit} have a simultaneous ordered spectral decomposition since the critical cone $K_\theta\big(\lm(X),\lm(Y)\big)$ can be often calculated rather easily. 
While this remains   an open question for now and will be a subject of our future research, we show by an example below the possible role that   the condition 
on $\Lm(Y)_{\al_m\al_m}$ and $U_{\al_m}^\top  HU_{\al_m}$ is playing  in the calculation of  the critical cone of spectral functions. Indeed, if 
the matrices $\Lm(Y)_{\al_m\al_m}$ and $U_{\al_m}^\top  HU_{\al_m}$ have a simultaneous ordered spectral decomposition, it is possible to 
show that the matrix $U_{\al_m}^\top  HU_{\al_m}$ has a block diagonal structure; see \eqref{spcr4} below and the discussion afterward to see why this can happen in the case $g=\dd_{\S^n_+}$.

\begin{Example}{\rm  Set $g=\dd_{\S^n_+}$.  Clearly, $g$ is a spectral function satisfying \eqref{spec} with $\th=\dd_{\R^n_+}$. 
 Take $Y\in N_{\S^n_+}(X)$ and observe from Proposition~\ref{subsp} that $Y=U\Diag(\lm(Y))U^\top$, where $\lm(Y)\in N_{\R^n_+}(\lm(X))$ and $U\in \O^n(X)\cap \O^n(Y)$. 
Our goal is to calculate $K_{\S^n_+}(X,Y)$ using the characterization of this cone from Proposition~\ref{crit}. Take $H\in K_{\S^n_+}(X,Y)$, assume that 
$\mu_1>\cdots>\mu_r$ are the distinct eigenvalues of $X$, and pick the constants  $\al_m$ for any $m=1,\ldots,r$  from \eqref{index}. By Proposition~\ref{crit}, 
we conclude that  $\lm'(X;H)\in K_ {\R^n_+}\big(\lm(X),\lm(Y)\big)$ and that the matrices $\Lm(Y)_{\al_m\al_m}$ and $U_{\al_m}^\top  HU_{\al_m}$ have a simultaneous ordered spectral decomposition for any $m=1,\ldots,r$.
The former is equivalent to the conditions 
\begin{equation}\label{spcr1}
w\in T_{\R^n_+}\big(\lm(X))\quad \mbox{and}\quad  \quad \big\la \lm(Y), w\big\ra=0\quad \mbox{with}\;\;w=(w_1,\ldots,w_n):=\lm'(X;H).
\end{equation}
Moreover, the inclusion $\lm(Y)\in N_{\R^n_+}(\lm(X))$ amounts to 
\begin{equation}\label{spcr2}
\sum_{i=1}^n \lm_i(X)\lm_i(Y)  =0, \quad \lm(X)=\big(\lm_1(X),\ldots,\lm_n(X)\big)\in \R^n_+, \quad \lm(Y)=\big(\lm_1(Y),\ldots,\lm_n(Y)\big)\in \R^n_-.
\end{equation}
Define the index sets 
$$
\kappa_X:=\big\{ i\in \{1,\ldots,n\}\, \big|\, \lm_i(X)>0\big\}\quad \mbox{and}\quad \tau_X:=\big\{ i\in \{1,\ldots,n\}\, \big|\, \lm_i(X)=0\big\},
$$
and 
$$
\kappa_Y:=\big\{ i\in \{1,\ldots,n\}\, \big|\, \lm_i(Y)=0\big\}\quad \mbox{and}\quad \tau_Y:=\big\{ i\in \{1,\ldots,n\}\, \big|\, \lm_i(Y)<0\big\}.
$$
It is easy to see from the definition of the index set $\al_r$ in \eqref{index} that $\tau_X=\al_r$ and from 
  \eqref{spcr2} that 
\begin{equation}\label{spcr3}
\kappa_X \subset \kappa_Y\quad \mbox{and}\quad \tau_Y\subset \tau_X,
\end{equation}
and to conclude  from \eqref{spcr1} and \eqref{spcr2} that 
\begin{equation}\label{spcr5}
w_i \in 
\begin{cases}
\R & \mbox{if}\;\; i\in \kappa_X,\\
\R_+& \mbox{if}\;\; i\in \tau_X\setminus \tau_Y,\\
\{0\}& \mbox{if}\;\; i\in \tau_Y.\\
\end{cases}
\end{equation}
Take $m\in \{1,\ldots,r-1\}$ and observe from the first inclusion in \eqref{spcr3} that $\Lm(Y)_{\al_m\al_m}=0$.
In this case, we will not benefit further from the fact that the matrices $\Lm(Y)_{\al_m\al_m}$ and $U_{\al_m}^\top  HU_{\al_m}$ have a simultaneous ordered spectral decomposition.
It remains to take a closer look into the case $m=r$. Since $\Lm(Y)_{\al_r\al_r}$ and $U_{\al_r}^\top  HU_{\al_r}$ have a simultaneous ordered spectral decomposition, 
we find an orthogonal matrix $  Q_r\in  \O^{|\al_r|}$ such that 
\begin{equation}\label{spcr4}
\Lm(Y)_{\al_r\al_r}=  Q_r\Lm(Y)_{\al_r\al_r}  Q_r^\top\quad \mbox{and}\quad  U_{\al_r}^{\top} H U_{\al_r}=  Q_r\Lambda(U_{\al_r}^\top  H U_{\al_r})  Q_r^\top.
\end{equation}
Similar to  \eqref{index}, define the index sets $\{\rho_r^{\nu}\}_{\nu=1}^\ell$, where $\ell$ in the number of distinct eigenvalues of $\Lm(Y)_{\al_r\al_r}$, by 
$$
\begin{cases}
\lm_i(Y)=\lm_j(Y)& \mbox{if}\;\; i,j\in \rho_r^{\nu}\\
\lm_i(Y)>\lm_j(Y)& \mbox{if}\;\; i\in \rho_r^{\nu}, \;\; j\in \rho_r^k \;\; \mbox{with}\;\; \nu<k.
\end{cases}
$$
Thus, since $\rho_r^j\subset \al_r=\tau_X$ for any $j\in\{1,\ldots,\ell\}$,  we deduce from the third condition in \eqref{spcr2} that 
\begin{equation}\label{spcr6}
\rho_r^1=\tau_X\cap  \kappa_Y\quad \mbox{and}\quad \bigcup_{j=2}^\ell \rho_r^j= \tau_X \cap \tau_Y=\tau_Y.
\end{equation}
It is not hard to see from the first equality in \eqref{spcr4} (see \cite[Proposition~2.4]{d12} for more detail)  that $  Q_r$ has a block diagonal representation as 
$$
Q_r=\Diag(Q_r^1,\ldots, Q_r^{\ell})=
\begin{pNiceMatrix}
Q_r^1  &0             & \Cdots        & 0         \\
0         & Q_r^2     & \Ddots        & \Vdots \\
\Vdots &\Ddots     & Q_r^{\ell-1} &  0        \\
0         &\Cdots     & 0                & Q_r^{\ell}
\CodeAfter \line{2-2}{3-3} 
\end{pNiceMatrix}
\quad \mbox{with}\;\; Q_r^j\in  \O^{|\rho_r^j|}, \;\; j=1,\ldots,\ell. 
$$
This, coupled with the second equality in \eqref{spcr4}, implies  that $U_{\al_r}^{\top} H U_{\al_r}$ has a similar block diagonal representation as
$$
U_{\al_r}^{\top} H U_{\al_r}=\Diag(H_r^1,\ldots,H_r^\ell)=
\begin{pNiceMatrix}
H_r^1  &0             & \Cdots        & 0         \\
0         & H_r^2     & \Ddots        & \Vdots \\
\Vdots &\Ddots     & H_r^{\ell-1} &  0        \\
0         &\Cdots     & 0                & H_r^{\ell}
\CodeAfter \line{2-2}{3-3} 
\end{pNiceMatrix}
\quad \mbox{with}\;\; H_r^j\in  \S^{|\rho_r^j|}, \;\; j=1,\ldots,\ell. 
$$
A direct calculation shows that $H_r^j=U_{\rho_r^j}^{\top} H U_{\rho_r^j}$ for any $ j=1,\ldots,\ell$. Moreover, it follows from \eqref{spcr5},  \eqref{spcr6},
and the definition of $w_i$ from \eqref{spcr1}  that 
$$
U_{\rho_r^1}^{\top} H U_{\rho_r^1}\in \S^{|\rho_r^1|}_+\quad \mbox{and}\quad U_{\rho_r^j}^{\top} H U_{\rho_r^j}=0\;\;\mbox{ for any}\;\;  j=2,\ldots,\ell,
 $$
which in turn leads us to the representation 
\begin{equation}\label{last}
U_{\al_r}^{\top} H U_{\al_r}=
\begin{pNiceMatrix}[margin,hvlines]
\Block{2-2} {U_{\rho_r^1}^{\top} H U_{\rho_r^1}}    &&    \Block{2-1} {\bigzero} \\
 \hspace*{1cm} &  \\
\Block{1-2} {\bigzero}     &&  \Block{1-1} {\bigzero} 
\end{pNiceMatrix}.
\end{equation}
This gives us  the inclusion 
\begin{equation}\label{last2}
K_{\S^n_+}(X,Y) \subset \Big\{ H\in \S^n\,\Big|\, U_{\rho_r^1}^{\top} H U_{\rho_r^1}\in \S^{|\rho_r^1|}_+, \; \; U_{\tau_Y}^{\top} H U_{\tau_Y}=0,\;\;  U_{\rho_r^1}^{\top} H U_{\tau_Y}=0\Big\}.
\end{equation}
 We claim now that the inclusion above  becomes equality. To prove it, take a matrix $H$ from the right-hand side of the above inclusion.
To justify $H\in K_{\S^n_+}(X,Y)$, we first show that $\lm'(X;H)\in K_ {\R^n_+}\big(\lm(X),\lm(Y)\big)$, which is equivalent to proving \eqref{spcr1}.
By the selection of $H$, the components of the vector $w=\lm'(X;H)$ enjoy the properties in \eqref{spcr5}, since $\al_r=\tau_X$ and $\rho_r^1=\tau_X\cap  \kappa_Y=\tau_X\setminus \tau_Y$
due to \eqref{spcr3}. Now, it is not hard to see that $w$ satisfies all the conditions in \eqref{spcr1},  confirming the inclusion $\lm'(X;H)\in K_ {\R^n_+}\big(\lm(X),\lm(Y)\big)$.  To finish the proof, 
it suffices, according to Proposition~\ref{crit}, to demonstrate that  the matrices $\Lm(Y)_{\al_m\al_m}$ and $U_{\al_m}^\top  HU_{\al_m}$ have a simultaneous ordered spectral decomposition for any $m=1,\ldots,r$.
Take $m\in \{1,\ldots, r-1\}$ and observe that if $i\in \al_m\subset \kappa_X$, it follows from the first inclusion in \eqref{spcr3} that $\lm_i(Y)=0$. This implies    that $\Lm(Y)_{\al_m\al_m}=0$ for any such an $m$,
which in turn tells us that $\Lm(Y)_{\al_m\al_m}$ and  $U_{\al_m}^\top  HU_{\al_m}$ have a simultaneous ordered spectral decomposition. It remains to consider the case $m=r$. 
We know from the selection of $H$ that $U_{\al_r}^\top  HU_{\al_r}$ has the representation in \eqref{last}. According to \eqref{spcr6}, the diagonal matrix $\Lm(Y)_{\al_r\al_r}$
has a similar block structure as \eqref{last}, given by 
$$
\Lm(Y)_{\al_r\al_r}=
\begin{pNiceMatrix}[margin,hvlines]
\Block{2-2} {\bigzero}    &&    \Block{2-1} {\bigzero} \\
 \hspace*{1cm} &  \\
\Block{1-2} {\bigzero}     &&  \Block{1-1} {\Lm(Y)_{\gamma\gamma}}
\end{pNiceMatrix}
\quad \mbox{with} \;\; \gamma:=\al_r\setminus \rho_r^1.
$$
Taking into account this representation and \eqref{last} tells us that $\Lm(Y)_{\al_r\al_r}$ and  $U_{\al_r}^\top  HU_{\al_r}$ have a simultaneous ordered spectral decomposition, and hence finishes 
the proof of the opposite inclusion in \eqref{last2}. We should add here that the same description as \eqref{last2} for $K_{\S^n_+}(X,Y)$ was obtained in \cite{cs} using a different approach
and without appealing to a characterization of the latter cone obtained in Proposition~\ref{crit}. 

Note that the analysis above for the case of $\dd_{\S^n_+}$ clearly illustrates the essential role that the simultaneous ordered spectral decompositions of 
$\Lm(Y)_{\al_m\al_m}$ and $U_{\al_m}^\top  HU_{\al_m}$   play in finding the critical cone of spectral functions.  
To shed more light into the role of  the later condition,  consider the case $n=3$, $X=\Diag(1,0,0)$ and $Y=\Diag(0,0,-1)$. In this case, we get 
$r=2$, $\al_1=\{1\}$, and $\al_2=\{2,3\}$. Moreover, we have $\kappa_X=\{1\}$, $\tau_X=\{2,3\}$, $\kappa_Y=\{1,2\}$, $\tau_Y=\{3\}$,   and 
$\rho_r^1=\{2\}$. According to \eqref{last2}, the symmetric matrix $H$ belongs to $K_{\S^n_+}(X,Y)$ if and only if it has 
a representation of the form 
$$
\begin{pmatrix}
\bigstar & \bigstar & \bigstar\\
\bigstar & a & 0\\
\bigstar & 0 & 0
\end{pmatrix},
$$
where $a\in \R_+$ and the $\bigstar$ positions can be filled with any real number. This shows the matrix 
$$
H=\begin{pmatrix}
1 & 0 & 0\\
0 & 1/2 & -1/2\\
0 & -1/2 & 0
\end{pmatrix}
$$
doesn't belong to $K_{\S^3_+}(X,Y) $ since $H_{23}$ and $H_{32}$ are not zero. To elaborate more  on why this happens, observe first that 
$U:=I\in \O^3(X)\cap \O^3(Y)$. Thus, we have 
$$
\lm'(X;H)=\big(\lm(U_{\al_1}^\top H U_{\al_1}), \lm(U_{\al_2}^\top H U_{\al_2})\big)=\big(\lm( H_{\al_1\al_1}), \lm(  H_{\al_2\al_2})\big)= (1,\frac{1+\sqrt{5}}{4},\frac{1-\sqrt{5}}{4}),
$$
which clearly belongs to $K_ {\R^3_+}\big(\lm(X),\lm(Y)\big)$. This can be justified via the equivalent description of $K_ {\R^3_+}\big(\lm(X),\lm(Y)\big)$
in \eqref{spcr1}. However, it is possible to demonstrate that $\Lm(Y)_{\al_2\al_2}$ and $U_{\al_2}^\top  HU_{\al_2}$ don't have a simultaneous ordered spectral decomposition by 
showing that 
\begin{equation}\label{last3}
\big\la \Lm(Y)_{\al_2\al_2},U_{\al_2}^\top  HU_{\al_2}\big\ra\neq \big\la \lm( \Lm(Y)_{\al_2\al_2}), \lm(U_{\al_2}^\top  HU_{\al_2})\big\ra.
\end{equation}
This, indeed, results from the fact that Fan's inequality in \eqref{fani} becomes equality if and only if the matrices therein have a simultaneous ordered spectral decomposition.
To prove \eqref{last3}, we deduce from $ \Lm(Y)_{\al_2\al_2}=\Diag(0,-1)$ and $U_{\al_2}^\top  HU_{\al_2}=H_{\al_2\al_2}$ that 
$$
\big\la \Lm(Y)_{\al_2\al_2},U_{\al_2}^\top  HU_{\al_2}\big\ra= \tr \Big(
\begin{pmatrix}
0 & 0 \\
0 & -1
\end{pmatrix}
\begin{pmatrix}
1/2 & -1/2\\
-1/2 & 0
\end{pmatrix}
\Big)=0.
$$
On the other hand, we have 
$$
 \big\la \lm( \Lm(Y)_{\al_2\al_2}), \lm(U_{\al_2}^\top  HU_{\al_2})\big\ra= \frac{\sqrt{5}-1}{4},
$$
which confirms \eqref{last3}, and hence tells us that the main reason for $H\notin K_ {\S^3_+}\big(X,Y\big) $ is the failure of ensuring a
simultaneous ordered spectral decomposition for $\Lm(Y)_{\al_2\al_2}$ and $U_{\al_2}^\top  HU_{\al_2}$. 
}
\end{Example}

As mentioned before, we can partition any vector $p\in \R^n$ into  $(p_{\al_1},\ldots, p_{\al_r})$ with $\al_m$, $m=1,\ldots,r$,  taken from \eqref{index}.
Pick $m\in \{1,\ldots,r\}$ and recall from \eqref{index2} that the index set $\al_m=\cup_{i=1}^{\rho_m}\beta_i^m$. This allows to  partition further  $p_{\al_m}$ into $(y_{\beta_1^m},\ldots,y_{\beta_{\rho_m}^m})$, where 
$y_{\beta_i^m}\in \R^{|\beta_i^m|}$ for any $i=1,\ldots,\rho_m$. In summary, we can equivalently write $p$ as 
\begin{equation}\label{nrep}
(y_{\beta_1^1},\ldots,y_{\beta_{\rho_1}^1},\ldots, y_{\beta_1^r},\ldots,y_{\beta_{\rho_r}^r}),
\end{equation}
where  $r$, taken from \eqref{index},  and $\rho_m$, taken from \eqref{index2},  stand for the number of distinct eigenvalues of $X$ and $U_{\al_m}^\top H U_{\al_m}$, respectively. 
Thus, the representation of $p$  in \eqref{nrep} is associated with the permutation matrices in $\P^n_{X,H}$, defined prior to  Proposition~\ref{psymm}, as a subset of $\P^n$.
In fact, any permutation matrix $Q\in \P^n_{X,H}$ has a representation of the form $\Diag(B^1_1,\ldots,B^1_{\rho_1},\ldots, B^r_1,\ldots,B^r_{\rho_r})$,  where $B^m_j \in \R^{|\beta_j^m|\times| \beta_j^m|}$ is a permutation matrix for any $j\in\{1,\ldots,\rho_m\}$
and $m\in \{1,\ldots,r\}$. Denote by $\R^n_{\dn}$ the set of all vectors $(x_1,\ldots,x_n)$ such that $x_1\ge \cdots\ge x_n$.

\begin{Proposition}\label{imp}Assume that the spectral function $g=\th\circ \lm$ in \eqref{spec} is     lsc and convex and that $Y\in \sub g(X)$ and $H\in K_g(X,Y)$.
If $\th$ is parabolically regular at $\lm(X)$ for $\lm(Y)$, then the following properties hold.
 \begin{itemize}[noitemsep,topsep=0pt]
 \item [{\rm (a)}]   There exists $z\in \R^n$, which has a representation in the form in \eqref{nrep} with $y_{\beta_i^m}\in \R^{|\beta_i^m|}_{\dn}$ for any $i\in \{1,\ldots,\rho_m\}$ and $m\in \{1,\ldots,r\}$, satisfying 
\begin{equation}\label{prn}
\d^2 \th(\lm(X),\lm(Y))( \lambda' (X ; H ))= \d^2 \th  (\lm(X) ) ( \lambda' (X ; H ) \verll  z  )-  \la \lm(Y),z\ra.
\end{equation}
 \item [{\rm (b)}]  There exists a matrix $\Hat W\in \S^n$ such that  $\lambda^{''} (X ; H , \Hat W) =z$, where $z$ comes from {\rm(}a{\rm)}.
\end{itemize}
\end{Proposition}

\begin{proof} We deduce from $Y\in \sub g(X)$ and Proposition~~\ref{subsp} that $\lm(Y)\in \sub \th  (\lm(X) )$. Also it follows from 
$H\in K_g(X,Y)$ and Proposition~\ref{crit} that $\lm'(X;H)\in K_\theta\big(\lm(X),\lm(Y)\big)$. Employing now Proposition~\ref{parregularity}
ensures the existence of $p\in \R^n$ satisfying \eqref{prn}. As explained above, any such a vector  $p$ has a representation in the form of  \eqref{nrep} with $y_{\beta_{i}^m}\in \R^{|\beta_i^m|}$ 
for any $i\in \{1,\ldots,\rho_m\}$ and $m\in \{1,\ldots,r\}$. We are going to show that we can find a vector $p$ with the representation in \eqref{nrep} such that  $y_{\beta_{i}^m}\in \R^{|\beta_i^m|}_{\dn}$ 
for any $i\in \{1,\ldots,\rho_m\}$ and $m\in \{1,\ldots,r\}$, meaning the components of each $y_{\beta_{i}^m}$ have  nonincreasing order.  To this end, pick $m\in \{1,\ldots,r\}$ and $i\in \{1,\ldots,\rho_m\}$
and choose then    a $|\beta_i^m|\times |\beta_i^m|$ permutation matrix $B^m_i$ such that $q_{\beta_{i}^m}:=B^m_i y_{\beta_{i}^m}\in  \R^{|\beta_i^m|}_{\dn}$. Set $Q:=\Diag(B_1^1,\ldots,B_{\rho_1}^1,\ldots,B_1^r,\ldots,B_{\rho_r}^r)$
and observe that  $Q\in \P^n_{X,H}$. Moreover, let 
\begin{equation}\label{nz}
z:=(q_{\beta_{1}^1},\ldots,q_{\beta_{\rho_1}^1},\ldots, q_{\beta_{1}^r},\ldots,q_{\beta_{\rho_r}^r}).
\end{equation}
Clearly, we have $z=Qp$. It follows from Proposition~\ref{psymm} that 
$$
 \d^2 \th (\lambda(X))\big( \lambda' (X ; H ) \verlm z)= \d^2 \th (\lambda(X))\big( \lambda' (X ; H ) \verlm p).
 $$
 We claim now that $ \la \lm(Y),p\ra \le  \la \lm(Y),z\ra$. To justify it,  suppose that 
 $$
 (\lm(Y)_{\beta_{1}^1},\ldots,\lm(Y)_{\beta_{\rho_1}^1},\ldots, \lm(Y)_{\beta_{1}^r},\ldots,\lm(Y)_{\beta_{\rho_r}^r})
 $$
 is a partition of the vector $\lm(Y)$ corresponding to \eqref{nrep}. Note that $\lm(Y)_{\beta_{i}^m}\in  \R^{|\beta_i^m|}_{\dn}$ for any $i\in \{1,\ldots,\rho_m\}$ and $m\in \{1,\ldots,r\}$.
 Thus, we get 
 $$
 \la \lm(Y),p\ra =\sum_{m=1}^r \sum_{i=1}^{\rho_m} \la \lm(Y)_{\beta_{i}^m},y_{\beta_{i}^m}\ra\le  \sum_{m=1}^r \sum_{i=1}^{\rho_m} \la \lm(Y)_{\beta_{i}^m},q_{\beta_{i}^m} \ra= \la \lm(Y),z\ra,
 $$
where the inequality is a consequence of the Hardy-Littlewood-P\'olya inequality (cf. \cite[Proposition~1.2.4]{bl}). Set $\ph(x)= \d^2 \th  (\lm(X) ) ( \lambda' (X ; H ) \verll  x  )-  \la \lm(Y),x\ra$ for
any $x\in \R^n$ and observe from Proposition~\ref{parregularity} that $p$ is a minimizer of $\ph$. But we showed above that $\ph(z)\le \ph(p)$, which tells us that $z$ is also a minimizer of 
$\ph$. Thus, we arrive at $\ph(z)=\ph(p)$, which implies that \eqref{prn} holds for $z$. This proves (a).

Turning now to the proof of (b), pick the vector $z$ from \eqref{nz}. We can equivalently write via the index sets   $\al_m$, $m=1,\ldots,r$, from  \eqref{index} that  
\begin{equation}\label{repz}
 z=(z_{\al_1},\ldots,z_{\al_r})\quad \mbox{with}\quad z_{\al_m}=(q_{\beta_{1}^m},\ldots,q_{\beta_{\rho_m}^m})\in \R^{|\al_m|}\quad \mbox{for all}\;\; m\in \{1,\ldots,r\}.
\end{equation}
 Take   the $|\al_m|\times |\al_m|$ matrix $Q_m$, $m=1,\ldots,r$, from  \eqref{sdt1}  and consider the $n\times n$ block diagonal matrix 
\begin{equation}\label{mata}
 A=\Diag\big(Q_1\Diag(z_{\al_1})Q_1^\top,\ldots,Q_r\Diag(z_{\al_r})Q_r^\top\big).
\end{equation}
We claim that there exists a matrix $\Hat W\in \S^n$ such that for any $m=1,\ldots,r$ the relationship 
\begin{equation}\label{what}
U_{\al_m}^\top  \Hat W U_{\al_m}=    U_{\al_m}^\top \big(2H (X- \mu_{m}I)^{\dagger} H + UAU^\top\big) U_{\al_m}
\end{equation}
holds, where $\mu_1>\cdots>\mu_r$ are the distinct eigenvalues of $X$ and $U\in \O^n(X)$.
Indeed, to find such a matrix $\Hat W$, let $W\in \S^n$ and  set $\Hat W=UWU^\top$ in the above equality. This, coupled with \eqref{dfa}, leads us to  
$$
W_{\al_m\al_m}=U_{\al_m}^\top  UWU^\top U_{\al_m}=U_{\al_m}^\top  \Hat W U_{\al_m}=   U_{\al_m}^\top \big(2H (X- \mu_{m}I)^{\dagger} H +UAU^\top\big) U_{\al_m}
$$  
for all  $m=1,\ldots,r$. Define the matrix $W$ as the block diagonal matrix $\Diag(W_{\al_1\al_1},\ldots,W_{\al_r\al_r})$  from which we can  obtain the claimed matrix $\Hat W$. 
Suppose now that $i\in \{1,\ldots,n\}$. By \eqref{dex2}, there are  $m\in \{1,\ldots,r\}$ and $j\in \{1,\ldots,\rho_m\}$ such that $i\in \al_m$ and $\ell_i\in \beta_j^m$.
According to \eqref{paralam}, we have 
\begin{eqnarray} 
\lambda_i^{''} (X; H, \Hat W)&=& \lambda_{\ss \l'_i} \Big( (Q_m)_ {\beta^m_j}^\top  \big( U_{\al_m}^\top ( \hat W + 2  H (\mu_m I  - X)^{\dagger} H ) U_{\al_m} \big) (Q_m)_ {\beta^m_j}\Big)\nonumber\\
&=&  \lambda_{\ss \l'_i} \Big( (Q_m)_ {\beta^m_j}^\top  \big( U_{\al_m}^\top  UAU^\top  U_{\al_m} \big) (Q_m)_ {\beta^m_j}\Big)\nonumber\\
&=&  \lambda_{\ss \l'_i} \Big( (Q_m)_ {\beta^m_j}^\top  Q_m\Diag(z_{\al_m})Q_m^\top  (Q_m)_ {\beta^m_j}\Big)\nonumber\\
&=&  \lambda_{\ss \l'_i} \big(\Diag(q_{\beta_{j}^m})\big)\label{sur},
\end{eqnarray}
where the last two equalities result from \eqref{dfa}. 
Consider now a partition of $\lambda^{''} (X; H, \Hat W)$ corresponding to \eqref{nz} as 
$$
(\eta_{\beta_{1}^1},\ldots,\eta_{\beta_{\rho_1}^1},\ldots, \eta_{\beta_{1}^r},\ldots,\eta_{\beta_{\rho_r}^r}).
$$
Thus, it follows from \eqref{sur} and the definition $\ell'_i$ that   $$\eta_{\beta_{j}^m}= \lambda  \big(\Diag(q_{\beta_{j}^m})\big)=q_{\beta_{j}^m}$$ for any $j\in \{1,\ldots,\rho_m\}$ and $m\in \{1,\ldots,r\}$ since $q_{\beta_{j}^m} \in  \R^{|\beta_j^m|}_{\dn}$.
This confirms that $\lambda^{''} (X; H, \Hat W)=z$ and hence completes the proof of (b).
\end{proof}

We are now ready to characterize parabolic regularity of spectral functions. We begin with the following result in which we 
provide a sufficient condition to calculate the domain of the second subderivative of spectral functions.

\begin{Proposition}\label{calss8}
Assume that the spectral function $g=\th\circ \lm$ in \eqref{spec} is   convex and that $Y\in \sub g(X)$.
Then, we have $\dom \d^2g(X,Y)\subset K_g(X,Y)$. Equality holds if, in addition,   $g$ is parabolically epi-differentiable at $X$ for any $H\in K_g(X,Y)$.
\end{Proposition}
\begin{proof} The claimed inclusion results from \cite[Proposition~2.1(ii)-(iii)]{ms}. To establish the second claim, it follows from 
parabolic epi-differentiability of $g$ at  $X$ for any $H\in K_g(X,Y)$ that $  \dom  \d^2  g(X)( H\verll \; . \; ) \neq \emptyset$.
This, coupled with \cite[Proposition~3.4]{ms}, confirms that $\dom \d^2g(X,Y)= K_g(X,Y)$ and hence completes the proof.
\end{proof}

Note that the assumption of parabolic epi-differentiability of $g$ in the result above can be ensured  via Theorem~\ref{sochpar}(a)
by parabolic epi-differentiability of $\th$. An important class of functions for which this assumption automatically is satisfied is 
polyhedral functions; see \cite[Exercise~13.61]{rw}. 
This class of functions allows us to cover many   examples of spectral functions, which are important for applications.
We should add that  polyhedral functions that are symmetric were characterized in \cite[Proposition~1]{cd}. 

Our next result presents a characterization of parabolic regularity of spectral functions.

\begin{Theorem}[parabolic regularity of spectral functions]\label{calss}
Assume that the spectral function $g=\th\circ \lm$ in \eqref{spec} is   locally Lipschitz continuous with respect to its domain, lsc, and convex. 
Let $\mu_1>\cdots>\mu_r$ be the distinct eigenvalues of $X$.
 Then the following properties hold. 
 \begin{itemize}[noitemsep,topsep=0pt]
 \item [{\rm (a)}] If $Y\in \sub g(X)$  and $\th$ is parabolically regular at $\lm(X)$ for $\lm(Y)$ and parabolically epi-differentiable at $\lm(X)$, then $g$ is parabolically regular at $X$ for $Y$ and  for any $H\in K_g(X,Y)$ we have 
\begin{equation*}\label{lbss2}
\d^2g(X,Y)(H)= \d^2\th\big(\lm(X),\lm(Y)\big)\big(\lm'(X;H)\big) + 2\sum_{m=1}^{r}  \big\la \Lm(Y)_{\al_m\al_m}, U_{\al_m}^\top H ( \mu_{m}I- X)^{\dagger} H  U_{\al_m}\big\ra,
\end{equation*}
where   $\al_m$, $m=1,\ldots,r$, come from \eqref{index}, $U\in \O^n(X)\,\cap\, \O^n(Y)$, and $\Lm(Y)=\Diag(\lm(Y)\big)$.
 \item [{\rm (b)}] If  $g $ is parabolically epi-differentiable and parabolically regular at $X$, then $\th$ is parabolically regular at $\lm(X)$.
\end{itemize}
\end{Theorem}
\begin{proof} We begin with the proof of (a). To justify (a), it suffices by Proposition~\ref{parregularity}  to show that   for any  $H\in K_g(X,Y)$, we have 
\begin{equation}\label{prsp}
\d^2g(X,Y)(H)  =\inf_{W\in \S^n}\big\{ \d^2 g(X)( H\verll W)- \big\la Y,  W\big\ra\big\}.
\end{equation}
To this end, pick  $H\in K_g(X,Y)$ and deduce from \cite[Proposition~13.64]{rw} and  \eqref{sochpar1}, respectively, that    
\begin{eqnarray}\label{usss}
\d^2g(X,Y)(H) & \le &\inf_{W\in \S^n}\big\{ \d^2 g(X)( H\verll W)- \big\la Y,  W\big\ra\big\} \nonumber\\
&= &\inf_{W\in \S^n}\big\{ \d^2 \th \big(\lm(X)\big)\big( \lambda' (X ; H ) \verll \lambda^{''} (X ; H , W) \big)- \big\la Y,  W\big\ra\big\}.
\end{eqnarray}
Since $H\in K_g(X,Y)$ , it results from Proposition~\ref{crit} that the matrices  $\Lm(Y)_{\al_m\al_m}$ and $U_{\al_m}^\top  HU_{\al_m}$ have a simultaneous ordered spectral decomposition for any $m=1,\ldots,r$.
This means that there are matrices $\Hat Q_m\in \O^{|\al_m|}(\Lm(Y)_{\al_m\al_m})\cap \O^{|\al_m|}(U_{\al_m}^{\top} H U_{\al_m})$, $m=1,\ldots,r$,  such that 
\begin{equation}\label{sod}
\Lm(Y)_{\al_m\al_m}=\Hat Q_m\Lm(Y)_{\al_m\al_m}\Hat Q_m^\top\quad \mbox{and}\quad  U_{\al_m}^{\top} H U_{\al_m}=\Hat Q_m\Lambda(U_{\al_m}^\top  H U_{\al_m})\Hat Q_m^\top.
\end{equation}
Replace the matrices $  Q_m$ in the definition of the matrix $A$ in \eqref{mata} with $\Hat Q_m$, $m=1,\ldots,r$, and observe that the same conclusion can be achieved as the one in Proposition~\ref{imp}(b)
for the updated matrix $A$. In fact, the matrices $\Hat Q_m$ enjoy all the properties of $  Q_m$ together with the   relationships in \eqref{sod}, which are important for our argument below. 
Since $\th$ is parabolically regular at $\lm(X)$ for $\lm(Y)$, we conclude from Proposition~\ref{imp}(a) that there is $z\in \R^n$
with a representation in the form in \eqref{nz} with $q_{\beta_i^m}\in \R^{|\beta_i^m|}_{\dn}$ for any $i\in \{1,\ldots,\rho_m\}$ and $m\in \{1,\ldots,r\}$, satisfying \eqref{prn}.
According to Proposition~\ref{imp}(b), there exists a matrix $\Hat W\in \S^n$ such that  $\lambda^{''} (X ; H , \Hat W) =z$. 
 Employing now  \eqref{prn} and  \eqref{what},  and $Y=U\Lm(Y)U^\top$ and using a similar argument as \eqref{sdcy}, we arrive at 
\begin{eqnarray} 
  &  &\d^2 \th \big(\lm(X)\big)\big( \lambda' (X ; H ) \verll \lambda^{''} (X ; H , \Hat W) \big)- \big\la Y,\Hat W\big\ra \nonumber \\
&= & \d^2 \th \big(\lm(X)\big)\big( \lambda' (X ; H ) \verll  z \big) - \sum_{m=1}^{r}  \big\la \Lm(Y)_{\al_m\al_m}, U_{\al_m}^\top  \Hat W U_{\al_m} \big\ra \nonumber \\
&=& \d^2 \th(\lm(X),\lm(Y))( \lambda' (X ; H ))+  \la \lm(Y),z\ra\nonumber\\
&&-\sum_{m=1}^{r}  \big\la \Lm(Y)_{\al_m\al_m},  U_{\al_m}^\top \big(2H (X- \mu_{m}I)^{\dagger} H + UAU^\top\big) U_{\al_m}\big\ra \nonumber\\
&=&  \d^2 \th(\lm(X),\lm(Y))( \lambda' (X ; H ))+ 2\sum_{m=1}^{r}  \big\la \Lm(Y)_{\al_m\al_m},  U_{\al_m}^\top \big(H ( \mu_{m}I-X)^{\dagger} H \big) U_{\al_m}\big\ra \nonumber\\
&& + \la \lm(Y),z\ra -\sum_{m=1}^{r}  \big\la \Lm(Y)_{\al_m\al_m},  U_{\al_m}^\top \big( UAU^\top\big) U_{\al_m}\big\ra\label{usss2}.
\end{eqnarray}
By the definition of $A$ in \eqref{mata}, the equality in \eqref{dfa},  and the representation of the vector $z$ in \eqref{repz}, we obtain 
\begin{eqnarray*} 
\sum_{m=1}^{r}  \big\la \Lm(Y)_{\al_m\al_m},  U_{\al_m}^\top \big( UAU^\top\big) U_{\al_m}\big\ra&=& \sum_{m=1}^{r}  \big\la \Lm(Y)_{\al_m\al_m},  \Hat Q_m\Diag(z_{\al_m})\Hat Q_m^\top \big\ra\\
&=& \sum_{m=1}^{r}  \big\la \Hat Q_m^\top \Lm(Y)_{\al_m\al_m}  \Hat Q_m, \Diag(z_{\al_m}) \big\ra\\
&=& \sum_{m=1}^{r}  \big\la  \Lm(Y)_{\al_m\al_m}   , \Diag(z_{\al_m}) \big\ra=\la \lm(Y),z\ra,
\end{eqnarray*}
where the penultimate equality results from the   first relationship in \eqref{sod} and the last one is a consequence of $\Lm(Y)=\Diag(\lm(Y))$. This, coupled with \eqref{lbss},  \eqref{usss}, and \eqref{usss2}, brings us to 
\begin{eqnarray*}
&& \d^2 \th \big(\lm(X),\lm(Y)\big)\big( \lambda' (X ; H ) \big)+ 2\sum_{m=1}^{r}  \big\la \Lm(Y)_{\al_m\al_m}, U_{\al_m}^\top H ( \mu_{m}I- X)^{\dagger} H  U_{\al_m}\big\ra \\
 &\le &\d^2g(X,Y)(H)  \\
 &\le& \inf_{W\in \S^n}\big\{ \d^2 g(X)( H\verll W)- \big\la Y,  W\big\ra\big\} \\
 & \le &\d^2 \th \big(\lm(X)\big)\big( \lambda' (X ; H ) \verll \lambda^{''} (X ; H , \Hat W) \big)- \big\la Y,\Hat W\big\ra \\
  &= &  \d^2 \th \big(\lm(X),\lm(Y)\big)\big( \lambda' (X ; H ) \big)+ 2\sum_{m=1}^{r}  \big\la \Lm(Y)_{\al_m\al_m}, U_{\al_m}^\top H ( \mu_{m}I- X)^{\dagger} H  U_{\al_m}\big\ra.
\end{eqnarray*} 
These relationships clearly justify \eqref{prsp} and hence imply that $g$ is parabolically regular at $X$ for $Y$. Moreover, 
they confirm the claimed formula for the second subderivative of $g$ at $X$ for $Y$ for any $H\in K_g(X,Y)$ and hence complete the proof of (a).

Turning into the proof of (b), we need to show that $\th$ is parabolically regular at $\lm(X)$ for any $v\in \sub \th(\lm(X))$. 
To justify it, pick $v\in \sub \th(\lm(X))$ and $U\in \O^n(X)$ and deduce from \cite[Theorem~6]{l99} that $U\Diag (v) U^\top\in \sub g(X)$.
Since $g$ is convex and orthogonally invariant, the latter yields  $ \Diag(v)\in \sub g(\Lm(X))$, where $\Lm(X)=\Diag(\lm(X))$.
We claim that $g$ is parabolically regular at $\Lm(X)$ for $\Diag(v)$. 
To this end, set $Z:=U\Diag (v) U^\top$ and take  $H\in K_g(\Lm(X),\Diag(v))$. We conclude from \eqref{psd4} and the definition of the critical cone that $UHU^\top\in K_g(X,Z)$.
Thus, it also  follows from parabolic regularity of $g$ at $X$ for $Z$ and Proposition~\ref{parregularity} that there is 
$W\in \S^n$ such that 
\begin{equation}\label{pr90}
\d^2g(X,Z)(UHU^\top)  =  \d^2 g(X)( UHU^\top\verll W)- \big\la Z,  W\big\ra.
\end{equation}
 Since $g$ is orthogonally invariant, we get 
 \begin{eqnarray*} 
\d^2g(X,Z)(UHU^\top) & = & \liminf_{\substack{ t\dn 0 \\ H'\to UHU^\top }} \frac{g(X+tH')-g(X)-\la Z,H'\ra}{\sm t^2}\\
&=& \liminf_{\substack{ t\dn 0 \\ H'\to UHU^\top }} \frac{g(\Lm(X)+tU^\top H' U)-g(\Lm(X))-\la U^\top ZU,U^\top H'U\ra}{\sm t^2}\\
&\ge & \d^2g\big(\Lm(X),\Diag(v)\big)( H).
 \end{eqnarray*}
Similarly, we can show that $\d^2g(X,Z)(UHU^\top) \le  \d^2g\big(\Lm(X),\Diag(v)\big)(H)$, which leads us to $\d^2g(X,Z)(UHU^\top) =  \d^2g\big(\Lm(X),\Diag(v)\big)( H)$.
It follows from \eqref{psd5} that 
$$
\d^2 g(X)( UHU^\top\verll W)= \d^2 g(\Lm(X))(  H\verll U^\top WU).
$$
We also have $\big\la Z,  W\big\ra=\big\la \Diag(v),  U^\top WU\big\ra$. Combining these with \eqref{pr90} brings us to 
$$
 \d^2g\big(\Lm(X),\Diag(v)\big)(H)=\d^2 g(\Lm(X))(H\verll U^\top WU)- \big\la \Diag(v),  U^\top WU\big\ra.
$$
Since $H\in K_g(\Lm(X),\Diag(v))$ was taken arbitrarily and since  $\dom \d^2g\big(\Lm(X),\Diag(v)\big)\subset K_g(\Lm(X),\Diag(v))$ is satisfied because of Proposition~\ref{calss8}, 
we conclude from  Proposition~\ref{parregularity}     that $g$ is parabolically regular at $\Lm(X)$ for $\Diag(v)$.
Recall   that the symmetric function $\th$ satisfies \eqref{spec2}, which means that 
$\th$ can be represented as  $\th=g\circ F$ with  $F(x):= \Diag(x)$ for all $x\in \R^n$. 
We also deduce from the imposed assumption on  $\th$ and   \eqref{lipei}
that $g$ is locally Lipschitz continuous relative to its domain. According to \cite[Theorem~3.6]{mms1}, we get
$\sub \th(\lm(X))=\nabla F(\lm(X))^*\sub g(\Lm(X))$. It is not hard to see that for any $x\in \R^n$, $\nabla F(x)$ is a linear operator from $\R^n$ into $\S^n$, defined by $\nabla F(x) (y)=\sum_{i=1}^n y_i E_{ii}$ for any $y=(y_1,\ldots,y_n)\in \R^n$,
where   $E_{ii}$, $i=1,\ldots,n$, are the $n\times n$ matrix with $(i,i)$ entry equal to $1$ and elsewhere  equal to $0$.
This tells us that the adjoint operator $\nabla F(x)^*:\S^n\to \R^n$ has a representation in the form $\nabla F(x)^*B=\big(\tr(E_{11}B),\ldots,\tr(E_{nn}B)\big)$
for any $B\in \S^n$; see \cite[Example~1.8]{bec} for more details. Remember that $v\in \sub \th(\lm(X))$ and $ \Diag(v)\in \sub g(\Lm(X))$. Thus, we have 
$
\nabla F(\lm(X))^*\Diag(v)=v.
$
On the other hand, it follows from the proof of Theorem~\ref{sochpar}(b) that parabolic epi-differentiability of $g$ at $X$ for $UHU^\top\in \dom \d g(X)$ implies that of $g$
at $\Lm(X)$ for $H$. Combining these and  \cite[Theorem~5.4]{ms} tells us  that $\th$ is parabolically regular at $\lm(X)$ for $v$ and hence completes the proof of (b).
\end{proof}

Given the spectral function $g$ in \eqref{spec}, it was shown in \cite[Proposition~10]{cdz} that if $\th$ is ${\cal C}^2$-cone reducible in the sense of \cite[Definition~6]{cdz},
then $g$ enjoys the same property. Note that ${\cal C}^2$-cone reducibility of $\th$ is strictly stronger assumption than parabolic regularity of this function, utilized in Theorem~\ref{calss},
as shown in \cite[Theorem~6.2]{mms2} and \cite[Example~6.4]{mms2}. Note also   we showed in Theorem~\ref{calss}(b) that parabolic regularity of $g$ yields that of $\th$; such a result
was not achieved for ${\cal C}^2$-cone reducibility in \cite{cdz}. 

In many important applications of the spectral function $g$ in \eqref{spec}, the symmetric function $\th$ is a polyhedral function. In this case, all the assumptions
in Theorem~\ref{calss} are satisfied automatically. Furthermore, the second subderivative of $g$ has a simple representation as demonstrated below.

\begin{Corollary} \label{poly}
Assume that $g:\S^n\to \oR$ has  the spectral representation in \eqref{spec}  with the symmetric function $\th$ being  polyhedral.  
If  $\mu_1>\cdots>\mu_r$ are the distinct eigenvalues of $X$ and $Y\in \sub g(X)$, then 
 $g$ is parabolically regular at $X$ for $Y$ and  for any $H\in \S^n$ we have 
\begin{equation*}\label{lbss2}
\d^2g(X,Y)(H)= \dd_{K_g(X,Y)}(H)+ 2\sum_{m=1}^{r}  \big\la \Lm(Y)_{\al_m\al_m}, U_{\al_m}^\top H ( \mu_{m}I- X)^{\dagger} H  U_{\al_m}\big\ra,
\end{equation*}
where   $\al_m$, $m=1,\ldots,r$, come from \eqref{index}, $U\in \O^n(X)\,\cap\, \O^n(Y)$, and $\Lm(Y)=\Diag(\lm(Y)\big)$.
\end{Corollary}
\begin{proof} It follows from  \cite[Exercise~13.61]{rw} and \cite[Example~3.2]{ms}, respectively, that $\th$ is parabolically epi-differentiable and parabolically  regular at $\lm(X)$.
By Theorem~\ref{calss}(a), the spectral function $g$ is parabolically regular at $X$ for $Y$. Moreover, we know from Proposition~\ref{calss8} that $\dom \d^2g(X,Y)=K_g(X,Y)$.
Take $H\in K_g(X,Y)$ and observe from Proposition~\ref{crit} that $\lm'(X;H)\in K_{\th}\big(\lm(X),\lm(Y)\big)$. Thus, we  obtain from \cite[Proposition~13.9]{rw} that 
$$
\d^2\th\big(\lm(X),\lm(Y)\big)\big(\lm'(X;H)\big)=\dd_{K_{\th} (\lm(X),\lm(Y) )}\big(\lm'(X;H)\big)=0.
$$
Employing now   Theorem~\ref{calss}(a) tells us that 
$$
\d^2g(X,Y)(H)= 2\sum_{m=1}^{r}  \big\la \Lm(Y)_{\al_m\al_m}, U_{\al_m}^\top H ( \mu_{m}I- X)^{\dagger} H  U_{\al_m}\big\ra\quad \mbox{for all}\;\;H\in K_g(X,Y).
$$
On the other hand, if $H\notin K_g(X,Y)$, we have $\d^2g(X,Y)(H)=\infty$ since $\dom \d^2g(X,Y)=K_g(X,Y)$. This proves the claimed formula for the second subderivative of $g$
at $X$ for $Y$.
\end{proof}

Note that the conjugate function of the parabolic subderivative of the spectral function $g$ in \eqref{spec}  with $\th$ therein being polyhedral
was recently calculated in \cite[Propositions~6 \& 10]{cd} by dividing a polyhedral function into two parts. In general, such a result gives us 
an upper bound for the second subderivative (cf. \cite[Proposition~13.64]{rw}). According to Proposition~\ref{parregularity}, parabolic regularity is, indeed, equivalent to saying that 
the latter conjugate function coincides with the second subderivative. We should add that parabolic regularity of $g$ was not discussed in \cite{cd} and so Corollary~\ref{lbss2}
can't be derived from the aforementioned results in \cite{cd}. 

We continue to apply the formula of the second subderivative, obtained in Corollary~\ref{poly}, for two important examples of spectral functions
and show how one can simplify the established formula for the second subderivative in  these cases.

\begin{Example}\label{ssex}{\rm
 \begin{itemize}[noitemsep,topsep=0pt]
 \item [{\rm (a)}]   Assume that $g:\S^n\to \R$ is defined by $g(X)=\lm_{\max}(X)$, where $\lm_{\max}$ stands for the maximum eigenvalue of $X$. 
 $g$ is a spectral function and satisfies the representation \eqref{spec} and  $\th(x)=\max\{x_1,\ldots,x_n\}$ with $x=(x_1,\ldots,x_n)\in \R^n$. 
 Take $Y\in \sub g(X)$ and observe from Proposition~\ref{subsp} that $Y=U\Diag(\lm(Y))U^\top$, where $\lm(Y)\in \sub \th(\lm(X))$ and $U\in \O^n(X)\cap \O^n(Y)$. 
Recall from \eqref{index} that $\al_1=\big\{i\in \{1,\ldots,n\}|\; \lm_i(X)=\lm_{\max}(X)\big\}$. It follows from \cite[Exercise~8.31]{rw} that 
 $$
 \sub \th(\lm(X))=\big\{ (t_1,\ldots, t_n)|\; \sum_{i=1}^nt_i=1, \;t_i \ge 0\;\;\mbox{for all}\; i\in \al_1, \;t_i=0\;\;\mbox{otherwise}\big\}.
 $$
 Taking into the consideration the formula for the second subderivative from Corollary~\ref{poly} and the notation therein and the  description of $ \sub g(\lm(X))$, we conclude that  
 $\Lm(Y)_{\al_m\al_m}=0$ for all $m\ge 2$. Moreover, we have 
 $$
 Y=U\Diag(\lm(Y))U^\top=\sum_{i=1}^n \lm_i(Y)U_iU_i^T=\sum_{i\in \al_1}  \lm_i(Y)U_iU_i^T=U_{\al_1}\Lm(Y)_{\al_1\al_1}U_{\al_1}^\top.
 $$
 Combining this and Corollary~\ref{poly}, we obtain for any $H\in \S^n$ that  
\begin{eqnarray*}
 \d^2g(X,Y)(H)&=& \dd_{K_g(X,Y)}(H)+2 \big\la \Lm(Y)_{\al_1\al_1}, U_{\al_1}^\top H ( \mu_{1}I- X)^{\dagger} H  U_{\al_1}\big\ra\\
 &=& \dd_{K_g(X,Y)}(H)+ 2\big\la Y, H ( \mu_{1}I- X)^{\dagger} H  \big\ra.
\end{eqnarray*}
This is the same formula, obtained in \cite[Theorem~2.1]{t2} for the second subderivative of the first leading eigenvalue of a symmetric matrix. 
Note that it was proven in \cite[Example~3.3]{ms} that all leading eigenvalues of a symmetric matrix is parabolically regular.
Also one can find  their second subderivatives  in \cite[Theorem~2.1]{t2}. Since  the leading eigenvalues, except the first one which is the maximum eigenvalue, are not convex, 
 Theorem~\ref{calss} and Corollary~\ref{poly} can't be utilized to cover them. That requires to extend the established theory in this section for subdifferentially regular functions in the sense of 
 \cite[Definition~7.25]{rw}, a task that we leave for our future research. 

 \item [{\rm (b)}]   Suppose that $g=\dd_{\S^n_-}$. As shown in Example~\ref{tans}, $g$ is a spectral function satisfying \eqref{spec} with $\th=\dd_{\R^n_-}$. 
 Take $Y\in N_{\S^n_-}(X)$ and observe from Proposition~\ref{subsp} that $Y=U\Diag(\lm(Y))U^\top$, where $\lm(Y)\in N_{\R^n_-}(\lm(X))$ and $U\in \O^n(X)\cap \O^n(Y)$. 
 If $\mu_1=\lm_1(X)<0$, it follows from $X\in \inte \S^n_-$ that $g$ is twice differentiable and $ \d^2g(X,Y)(H)=0$ for any $H\in \S^n$. Assume now that $\mu_1=\lm_1(X)=0$.
 Recall from \eqref{index} that $\al_1=\big\{i\in \{1,\ldots,n\}|\; \lm_i(X)=\mu_1\big\}$.  Thus we obtain 
 $$
 N_{\R^n_-}(\lm(X))=\big\{ (t_1,\ldots, t_n)|\; t_i\ge 0\;\;\mbox{for all}\; i\in \al_1, \;t_i=0\;\;\mbox{otherwise}\big\}.
 $$
 Arguing similar to (a) leads us to 
 $$
  \d^2\dd_{\S^n_-} (X,Y)(H)=\dd_{K_{\S^n_-}(X,Y)}(H)-2 \big\la Y, H  X^{\dagger} H  \big\ra\quad \mbox{for all}\; H\in \S^n.
 $$
 This formula was obtained perviously in \cite[Example~3.7]{ms} via a different approach. Similarly, we can show that 
 if $Y\in N_{\S^n_+}(X)$, the second subderivative of $\dd_{\S^n_+}$ at $X$ for $Y$ can be calculated by 
 $$
  \d^2\dd_{\S^n_+} (X,Y)(H)=\dd_{K_{\S^n_+}(X,Y)}(H)-2 \big\la Y, H  X^{\dagger} H  \big\ra\quad \mbox{for all}\; H\in \S^n.
 $$
\end{itemize}

}
\end{Example}

The second subderivative can be utilized to establish   second-order optimality conditions for different classes of optimization problems.
Doing so often requires obtaining a chain rule for the second subderivative, a task carried out in Theorem~\ref{calss} and Corollary~\ref{poly}.
Given a twice differentiable function $\ph:\S^n\to \R$ and a spectral function $g$, 
consider the optimization problem 
\begin{equation}\label{comp}
\mbox{minimize}\; \ph(X)+g(X)\quad \mbox{subject to}\;\; X\in \S^n.
\end{equation}
Below, we present a result in which second-order optimality conditions for this optimization problem  are established.
For simplicity, we are going to assume that $g$ has the assumed representation in Corollary~\ref{poly} but one can easily extend it for any $g$ satisfying the assumptions in Theorem~\ref{calss}. 
\begin{Theorem} Assume that $X$ is a feasible solution to \eqref{comp}, where the spectral function $g$ has  the representation \eqref{spec} with $\th$ therein being a polyhedral function. 
If $-\nabla \ph(X)\in \sub g(X)$, then the following second-order optimality conditions for \eqref{comp} are satisfied.

 \begin{itemize}[noitemsep,topsep=0pt]
 \item [{\rm (a)}]  If $X$ is a local minimizer of \eqref{comp}, then the second-order necessary condition
 $$
 \nabla^2\ph(X)(H,H)+2\sum_{m=1}^{r}  \big\la \Lm(Y)_{\al_m\al_m}, U_{\al_m}^\top H ( \mu_{m}I- X)^{\dagger} H  U_{\al_m}\big\ra \ge 0
 $$
 holds for all $H\in K_g(X,-\nabla \ph(X))$.
 
  \item [{\rm (b)}] The validity of the second-order condition
  $$
  \nabla^2\ph(X)(H,H)+2 \sum_{m=1}^{r}  \big\la \Lm(Y)_{\al_m\al_m}, U_{\al_m}^\top H ( \mu_{m}I- X)^{\dagger} H  U_{\al_m}\big\ra >0  
  $$
  for all $H\in K_g(X,-\nabla \ph(X))$ amounts to the existence of the constants $\ell\ge 0$ and $\ve>0$ for which the quadratic growth condition 
  $$
  \ph(X')+g(X') \ge \ph(X)+g(X)+\frac{\ell}{2}\|X'-X\|^2\quad \mbox{for all}\;\; X'\in \B_\ve(X)
  $$
  is satisfied. 
 \end{itemize}
\end{Theorem}
\begin{proof} It follows from \cite[Exercise~13.18]{rw} that 
$$
 \d^2(\ph+g)(X,0)(H)=   \nabla^2\ph(X)(H,H)+  \d^2g(X,-\nabla \ph(X))(H)
 $$ for any $H\in \S^n$.
Both claims in (a) and (b) then result immediately from \cite[Theorem~13.24]{rw} and Corollary~\ref{poly}. 
\end{proof}

Our next application is to provide sufficient conditions for twice epi-differentiability of spectral functions, a concept with important consequences in second-order variational analysis 
and parametric optimization. Recall from  \cite[Definition~13.6]{rw} that  a function $f:\X \to \oR$ is said to be {twice epi-differentiable} at $\bar x$ for $\ov\in\X$, with $f(\ox) $ finite, 
if the sets $\epi \Delta_t^2 f(\bar x , \ov)$ converge to $\epi \d^2 f(\bar x,\ov)$ as $t\downarrow 0$ in the sense of set convergence from \cite[Definition~4.1]{rw}, where `$\epi $' stands for the epigraph of a function.  This can be equivalently described via   
\cite[Proposition~7.2]{rw} that for every sequence $t_k\downarrow 0$ and every $w\in\X$, there exists a sequence $w_k \to w$ such that
\begin{equation*}\label{dedf}
\d^2 f(\bar x,\ov)(w) = \lim_{k \to \infty} \Delta_{t_k}^2 f(\ox , \ov)(w_k).
\end{equation*}

Twice epi-differentiability is a geometric notion of second-order approximation for extended-real-valued functions and was  defined by Rockafellar in \cite{r88}.
Its central role in second-order variational analysis, parametric optimization, and numerical algorithms has been demonstrated in \cite{rw,mms2,ms,hs}. 
It was observed recently in \cite[Corollary~5.5]{ms} that parabolic regularity of certain composite functions yields their twice epi-differentiability. A similar conclusion can be drawn for
 spectral functions as demonstrated below. 

\begin{Corollary}[twice epi-differentiability of spectral functions] \label{tepi}
Assume that the spectral function $g=\th\circ \lm$ in \eqref{spec} is   locally Lipschitz continuous with respect to its domain, lsc, and convex. 
If $Y\in \sub g(X)$  and $\th$ is parabolically regular at $\lm(X)$ for $\lm(Y)$ and parabolically epi-differentiable at $\lm(X)$, then  $g$ is  twice epi-differentiable at $X$ for $Y$.  

\end{Corollary}
\begin{proof} 
The claimed twice epi-differentiability of $g$ at  $X$ for $Y$ results directly  from \cite[Theorem~3.8]{ms} and Theorem~\ref{calss}. 
\end{proof}

Twice epi-differentiability of leading eigenvalues and the sum of largest eigenvalues of a symmetric matrix was established in \cite[Theorem~2.1]{t2}
using a different approach. Corollary~\ref{tepi} goes far beyond the framework in \cite{t2} to achieve twice epi-differentiability of spectral functions.
We, however, can't get this property for all leading eigenvalues, except the first one, from Corollary~\ref{tepi} since these spectral functions are not convex.
As explained in Example~\ref{ssex}(a), this can be accomplished if the established theory in this section is generalized for subdifferentially regular functions.
Note also that a characterization of directionally differentiability of the proximal mapping of spectral functions can be found in \cite[Theorem~3]{dst}.
Recall from \cite[Theorem~7.18]{bec} that if the spectral function $g$ in \eqref{spec} is lsc and convex, its proximal mapping can be calculated by
$$
\prox_g(X):=\mbox{argmin}_{W\in \S^n}\big\{g(W)+\sm \|W-X\|^2\big\}=U\Diag\big(\prox_\th(\lm(X))\big)U^\top,
$$
where $U\in \O^n(X)$. It follows from \cite[Theorem~3]{dst}   that $\prox_g$ is directionally differentiable at $X$ if and only if $\prox_\th$ is directionally differentiable at $\lm(X)$.
It also follows from \cite[Exercise~13.45]{rw} that twice epi-differentiability of $g$ at $X$ for $Y$ amounts to directional differentiability of  $\prox_g$ at $X+Y$.
Combining these, we can conclude that $g$ is twice epi-differentiability of $g$ at $X$ for $Y$ if and only if $\th$ enjoys the same property at $\lm(X)$ for $\lm(Y)$.
It is not clear yet to us whether such an equivalence can be achieved via our approach. Note that  our main result in this section provides the equivalence for 
parabolic regularity of $g$ and $\th$. It is worth mentioning here that parabolic regularity is strictly stronger than twice epi-differentiability and has no counterpart 
for the proximal mapping. Thus, Theorem~\ref{calss} can't  be derived from \cite[Theorem~3]{dst}.

We close this section by establishing a characterization of   twice semidifferentiability of the spectral function $g$ in \eqref{spec}
when the symmetric function $\th$ therein is convex. 
Recall from \cite[Exercise~13.7]{rw} that a function $f\colon\R^n\to\oR$ is called {\em twice semidifferentiable} at $\ox$ if 
there exists a continuous function $h$, which is positive homogeneous of degree $2$, and 
\begin{equation*}
f(x)=f(\ox)+\d f(\ox)(x-\ox)+\sm h(x-\ox)+o(\|x-\ox\|^2).
\end{equation*}
In this case, $h$ is called the {\em second semiderivative} of $f$ at $\ox$ and  is denoted by $\d^2 f(\ox)$. 

\begin{Corollary}\label{tsdi}
Assume that $X\in \S^n$ and   $\mu_1>\cdots>\mu_r$ are the distinct eigenvalues of $X$ that $\th:\R^n\to \R$ is a differentiable symmetric convex function.
 Then  
 $\th$ is twice semidifferentiable at $\lm(X)$  if and only if the spectral function  $g=\th\circ \lm$ is twice semidifferentiable  at $X$. Moreover, 
 in this case, we have 
\begin{equation*} 
\d^2g(X)(H)= \d^2\th\big(\lm(X)\big)\big(\lm'(X;H)\big) + 2\sum_{m=1}^{r}  \big\la \Lm(Y)_{\al_m\al_m}, U_{\al_m}^\top H ( \mu_{m}I- X)^{\dagger} H  U_{\al_m}\big\ra,
\end{equation*}
where   $\al_m$, $m=1,\ldots,r$, come from \eqref{index}, $U\in \O^n(X)\,\cap\, \O^n(Y)$,  $\Lm(Y)=\Diag(\lm(Y)\big)$ and $H\in \S^n$.

 \end{Corollary}

\begin{proof} Observe first that since $\th$ is convex and finite-valued, it is locally Lipschitz continuous on $\R^n$. 
Moreover, because $\th$ is differentiable, we deduce from \cite[Theorem~1.1]{l96} that $g$ is differentiable at $X$. 
Suppose first that $\th$ is twice semidifferentiable at $\lm(X)$. Since $\th$ is differentiable, it follows from twice semidifferentiability of $\th$ that 
the second subderivative of $\th$ coincides with its second semiderivative, namely 
$$
\d^2\th\big(\lm(X),\nabla \th(\lm(X))\big)\big(\lm'(X;H)\big)=\d^2\th\big(\lm(X)\big)\big(\lm'(X;H)\big)\quad \mbox{for all}\; H\in \S^n.
$$
This, combined with the formula of the second subderivative of $g$ in Theorem~\ref{calss}(a), tells us  that for any $H\in \S^n$, $\d^2g(X,\nabla g(X))(H)$ is always finite. 
By \cite[Example~4.7(d)]{ms}, $\th$ is parabolically epi-differentiable and parabolically regular at $\lm(X)$ for $\nabla \th(\lm(X))$. Thus, it results from Corollary~\ref{tepi} that 
$g$ is twice epi-differentiable at $X$ for $\nabla g(X)$. Combining these with \cite[Theorem~4.3]{r20} tells us that $g$ is twice semidifferentiable 
at $X$ and that $\d^2g(X,\nabla g(X))(H)=\d^2g(X)(H)$ for any $H\in \S^n$. The latter, coupled with Theorem~\ref{calss}(a),  proves the claimed formula for the second semiderivative of $g$ at $X$.
Conversely, assume that  $g$ is twice semidifferentiable at $X$. We know from  \eqref{spec2} that    the symmetric function $\th$  can be represented as  $\th=g\circ F$ with  $F(x):= \Diag(x)$ for all $x\in \R^n$. 
Since $F$ is always twice differentiable and $g$ is twice semidifferentiable at $X$, we conclude from \cite[Proposition~8.2(i)]{mms1} that $\th$ is twice semidifferentiable at $\lm(X)$, which completes the proof. 
\end{proof}

One can also show similar to the proof of Corollary~\ref{tsdi} that $\th$ has a quadratic expansion at $\lm(X)$ if and only if $\th\circ \lm$ enjoys the same property at $X$.
Note that it was shown  in \cite[Theorem~3.3]{ls} by Lewis and Sendov (see also \cite{dp} for a simplified proof) that twice differentiability of $g$ and $\th$ are also equivalent. 
Whether such a result can be derived from our established theory in this section remains an open question for our future research.

\section{Conclusion and Future Research} 

In this paper, we developed a second-order theory of generalized differentiation for spectral functions. Our results  rely heavily upon the metric subregualrity 
constraint qualification, which automatically holds for this setting. Our main focus was to characterize parabolic regularity of this class of functions when 
they are convex. Moreover,  we were able to calculate their second subderivative.

Our results   raise several questions for our future search. First and foremost is the extension of our established theory for subdifferentially regular spectral functions.
This will allow us to provide a unified umbrella under which all the available results for both convex and nonconvex spectral functions can be covered by our approach. 
Also, it is interesting to see whether a similar characterization of parabolic regularity can be achieved for twice epi-differentiability of spectral functions. Such a 
characterization can be obtained from \cite[Theorem~3]{dst}  for convex 
spectral functions. However, we can't   use \cite{dst} to obtain a similar characterization for nonconvex spectral functions. It is also important 
to see whether our results can be utilized to characterize twice differentiability of spectral functions, which was perviously obtained by Lewis and Sendov in \cite[Theorem~3.3]{ls}.


\end{document}